\documentclass[10pt]{article}

\usepackage[utf8]{inputenc}
\usepackage{graphicx}%% uncomment this for including figures
\usepackage{amsthm,amsmath,amsfonts,amssymb,amstext}
\usepackage{enumerate}
\usepackage{caption}
\usepackage{mathtools}
\usepackage{amscd}
\usepackage{fancyhdr}
\usepackage{dsfont}
\usepackage{cite}      %package that allows multiple citations inside \cite
\usepackage{tikz}
\usetikzlibrary{decorations.markings}
\usetikzlibrary{arrows.meta}
\usetikzlibrary{patterns}
\usetikzlibrary{backgrounds}
\usetikzlibrary{calc}
\usetikzlibrary{angles}
\usetikzlibrary{quotes}
\usetikzlibrary{intersections}
\usetikzlibrary{through}

\usepackage[colorlinks,citecolor=blue,urlcolor=blue]{hyperref}%% uncomment this for coloring bibliography citations and linked URLs

\theoremstyle{plain}
\newtheorem{teo}{Theorem}[section]
\newtheorem{coro}[teo]{Corollary}
\newtheorem{lema}[teo]{Lemma}
\newtheorem{prop}[teo]{Proposition}

\theoremstyle{remark}

\newtheorem{defi}[teo]{Definition}
\newtheorem{remark}[teo]{Remark}

\tikzset{
    set arrow inside/.code={\pgfqkeys{/tikz/arrow inside}{#1}},
    set arrow inside={end/.initial=>, opt/.initial=},
    /pgf/decoration/Mark/.style={
        mark/.expanded=at position #1 with
        {
            \noexpand\arrow[\pgfkeysvalueof{/tikz/arrow inside/opt}]{\pgfkeysvalueof{/tikz/arrow inside/end}}
        }
    },
    arrow inside/.style 2 args={
        set arrow inside={#1},
        postaction={
            decorate,decoration={
                markings,Mark/.list={#2}
            }
        }
    },
}
  \newcounter{constant}
  \newcommand{\newconstant}[1]{\refstepcounter{constant}\label{#1}}
  \newcommand{\uc}[1]{c_{\textnormal{\tiny \ref{#1}}}}
\newcommand{\I}{\mathds{1}} % indicator
\newcommand{\distr}{\ensuremath{\stackrel{\scriptstyle d}{=}}}
\newcommand{\supp}{\ensuremath{\text{\upshape supp }}}
\newcommand{\dist}{\ensuremath{\text{\upshape dist}}}
\newcommand{\Poi}{\text{\upshape Poi}}
\newcommand{\PPP}{\text{\upshape PPP}}

\newcommand{\per}{\text{\upshape per}}
\newcommand{\comp}{{\mathsf c}}
\newcommand{\tI}{\smash{\tilde{I}}}
\newcommand{\ustick}{u^{\mathsf{st}}_c}

\newcommand{\SLE}{\ensuremath{\text{\upshape SLE}}}
\newcommand{\Arm}{\ensuremath{\text{\upshape Arm}}}
\newcommand{\EArm}{\ensuremath{\text{\upshape EArm}}}
\newcommand{\Circ}{\ensuremath{\text{\upshape Circ}}}
\newcommand{\vCirc}{\ensuremath{\overline{\text{\upshape Circ}}}}

\newcommand{\cC}{\ensuremath{\mathcal{C}}}

\newcommand{\cE}{\ensuremath{\mathcal{E}}}
\newcommand{\cF}{\ensuremath{\mathcal{F}}}

\newcommand{\cH}{\ensuremath{\mathcal{H}}}

\newcommand{\cP}{\ensuremath{\mathcal{P}}}
\newcommand{\cQ}{\ensuremath{\mathcal{Q}}}

\newcommand{\cS}{\ensuremath{\mathcal{S}}}
\newcommand{\cT}{\ensuremath{\mathcal{T}}}

\newcommand{\cV}{\ensuremath{\mathcal{V}}}

\newcommand{\EE}{\ensuremath{\mathbb{E}}}

\newcommand{\NN}{\ensuremath{\mathbb{N}}}
\newcommand{\PP}{\ensuremath{\mathbb{P}}}

\newcommand{\RR}{\ensuremath{\mathbb{R}}}
\newcommand{\ZZ}{\ensuremath{\mathbb{Z}}}

\begin{document}

\title{Holder continuity of interfaces for scale-invariant Poisson stick soup}
\author{Augusto Teixeira%
\thanks{
IMPA, Estrada Dona Castorina 110, 22460-320, Rio de Janeiro, RJ, Brazil.
Email: \texttt{augusto@impa.br}}
\quad \qquad
Daniel Ungaretti%
\thanks{
Instituto de Matemática, Universidade Federal do Rio de Janeiro, RJ, Brazil.
Email: \texttt{daniel@im.ufrj.br}}}
\maketitle

\begin{abstract}
We study the interface of covered and vacant sets in the subcritical phase of a
scale-invariant Poisson stick soup on the plane. This model is a natural
candidate for scaling limit of some planar models and has connections with
long-range percolation on the plane with critical parameter $s=4$. We
analyze a family of exploration paths on boxes and prove tightness for this
family and Holder continuity for its limiting measures.

\vspace{0.5cm}\noindent {\it Keywords:} Continuum percolation; Poisson stick
    soup; scale-invariance; exploration paths
%60K35 - Interacting random processes; statistical mechanics type models; percolation theory
%82B43 - Percolation
%60G55 - Point Processes
\end{abstract}

\section{Introduction}
\label{sec:Introduction}

Percolation processes are one of the simplest statistical physics models that
present a phase transition. They can be interpreted as models for environments in
which connectivity is assumed to be random and thus provide a natural framework
for modelling phenomena such as the spread of an infection through a
population, of information through wireless networks and of forest fires.
When the underlying environment is discrete, percolation is
usually defined as a probability measure on state configurations of a fixed
infinite graph and when the environment is continuous, a usual approach
is to define the model via a Poisson Point Process.
For background on the more classical Bernoulli percolation
on infinite graphs we refer the reader to
books~\cite{grimmett2010percolation, bollobas2006percolation} and 
a reference for continuum percolation theory
is~\cite{meester1996continuum}.

In this paper we study a very symmetrical continuum percolation model on the
plane, which we call \textit{scale-invariant Poisson stick soup} (IPS).
It is part of a family of random stick soups we call
\textit{scale-homogeneous Poisson stick soup} (HPS). These models are defined
via a Poisson Point Process (PPP) on the space
\begin{equation*}
    S := \RR^{2} \times \RR^{+} \times (-\pi/2, \pi/2],
\end{equation*}
in which a point $s = (z, R, V) \in S$ is interpreted as a segment
(stick) centered at $z$ with length $2R$ and direction $V$, denoted by
$E_0(s)$.
Fixed a parameter $u > 0$ that controls the density of sticks and parameter
$\alpha > 0$ that governs the tails of the radius distribution, we define HPS $\xi$ as a
PPP with intensity measure
\begin{equation}
\label{eq:shpss_measure}
u \mu_{\alpha}
    := u \ \mathrm{d}z \otimes \alpha R^{-(1+\alpha)} \text{d}R \otimes
    \frac{1}{\pi}\mathrm{d}V.
\end{equation}
Taking $\alpha = 2$ corresponds to the IPS with intensity $u$.
In our nomenclature, the term `scale-homogeneous' stresses that
after applying a homothety of ratio $c >0$ to the process we obtain a process
with intensity given by $c^{2 - \alpha}$ times the original intensity, see
Proposition~\ref{prop:homogeneity_relation_PSS}.
Hence, when $\alpha = 2$ applying any homothety to $\xi$ produces a model
with the same law as $\xi$, rendering the model scale-invariant.

Measures $u \mu_{\alpha}$ defining HPS are related to the study of
long-range percolation on $\ZZ^{2}$ (although the notion of connectivity in
these models are different). This can be seen by
reparametrizing the measure $u \mu_{\alpha}$ by using the sticks endpoints
$x, y$ instead of $z, R, V$. In this case, we obtain a measure on
$\RR^4$ with intensity
$\mu'_{\beta, s}(\,\mathrm{d} x \,\mathrm{d} y)
    = \frac{\beta}{|x-y|^s} \,\mathrm{d} x \,\mathrm{d} y$, with $s=2+\alpha$
and $\beta= \frac{2^{\alpha} \alpha}{\pi} u$. This measure can be seen as a
continuous version of long-range percolation on $\ZZ^{2}$ and has played an
important part in some arguments regarding the model,
cf.~\cite{biskup2019sharp} and references therein. Notice IPS, the model
with $\alpha=2$, correponds to the critical value $s=4$ in long-range
percolation.

Every HPS model has sticks of all scales and the sticks are actually
dense on the plane. We refer to the region covered by sticks as 
the \textit{covered} set, $\cE$, and its complement as the
\textit{vacant} set, $\cV$.
Regarding percolation probabilities for $\cE$ and $\cV$, only
IPS behaves in a non-trivial way, due to its scale-invariance. Indeed, we
prove in Proposition~\ref{prop:percolation_when_alpha_neq_2} that
\begin{equation*}
\text{for $\alpha \neq 2$ and any $u > 0$ an HPS satisfies
$\PP(\cE \ \text{percolates}) = 1$ and $\PP(\cV \ \text{percolates}) = 0$.}
\end{equation*}
Scale-invariant planar models are close in spirit to random Cantor sets and
have fractal properties worth investigating. Nacu and
Werner~\cite{nacu2011random} study scale-invariant soups built from
distributions on compact planar curves and their Hausdorff dimension.
One appealing reason for studying scale-invariant models is
that they are natural candidates for scaling limits. Here, by scaling limit
of a planar model we mean a limiting object obtained when we contract space
through a homothety while adjusting the density of the model to obtain a
non-degenerate process. In~\cite{nacu2011random} there is special emphasis
on the \textit{Brownian loop-soup}, an object introduced
in~\cite{lawler2004brownian} that arises as scaling limit of
conformally invariant planar lattice systems.

Our interest on IPS is born from investigating scaling limits of the 
\textit{ellipses model}, a model introduced by the authors
in~\cite{teixeira_ungaretti2017ellipses}. The ellipses model is a
Boolean model in the plane with defects given by heavy-tailed random ellipses.
The ellipses are centered on a Poisson point process of intensity $u > 0$, they
have uniform direction and the size of their minor axis is always equal to one.
However, their major axis has distribution $\rho$ supported on $[1, \infty)$
and satisfying
\begin{equation*}
\newconstant{const:R_decay} % introduce const:R_decay
\uc{const:R_decay}^{-1} r^{-\alpha}
    \le \rho[r,\infty)
    \le \uc{const:R_decay} r^{-\alpha}, \ \text{for every} \ r \geq 1
\end{equation*}
for some positive constant $\uc{const:R_decay}$ and $\alpha > 0$. The behavior
of connection probabilities on these models depends essentially on parameters
$u$ and $\alpha$. Reference~\cite{teixeira_ungaretti2017ellipses} studies 
phase transition for percolation on these models when we vary these parameters.

When performing scaling limits on ellipses model one should expect to obtain
a soup of sticks, but we should specify in which sense we are taking a limit.
One could for instance consider weak convergence for the PPP's on $S$ that
represents the major axis of the ellipses. Alternatively, it might be
interesting to study convergence in Smirnov Schramm
topology~\cite{schramm2011scaling}, see Remark~\ref{rem:smirnov_schramm}.

Our main purpose in this paper is understanding connection properties of
IPS related to `vacant crossings' of boxes of $\RR^2$. For any fixed $u>0$, we
say that the \textit{box-crossing property} holds if for any fixed $k > 0$ and $l > 0$
\begin{equation}
\label{eq:vacant_crossing_bounded_away_IPS}
\delta
    \le \PP\biggl(
        \begin{array}{c}
        \text{there is a left-right crossing of the box} \\
        \text{$[0,kl]\times[0,l]$ that does not `cross' any stick}
        \end{array}
        \biggr)
    \le 1 - \delta,
\end{equation}
where $\delta = \delta(u, k) > 0$. Definition~\ref{defi:cross_segment_by_curve}
makes precise our notion of a curve to `cross' a stick. At an intuitive
level, the curve is allowed to even go along the stick without `crossing' it,
provided that it goes around it.
The existence of vacant crossings is only possible for small $u$ intensities, characterized by
\begin{prop}[Nacu and Werner~\cite{nacu2011random}]
\label{prop:vacant_crossing_boxes_IPS}
It holds that
\begin{equation*}
\label{eq:defi_bar_u}
\bar{u}
    := \sup \{u; 
        \text{$\xi$ with intensity $u$ satisfies the box-crossing
        property~\eqref{eq:vacant_crossing_bounded_away_IPS}}\} > 0.
\end{equation*}
Moreover, for $u \in (0, \bar{u})$ we have:
\begin{equation}
\label{eq:cV_and_cE_not_percolate_IPS}
\PP[\text{neither $\cV$ nor $\cE$ percolate}] = 1.
\end{equation}
\end{prop}
Proposition~\ref{prop:vacant_crossing_boxes_IPS} is a reformulation
of results from \cite{nacu2011random}. For convenience of the reader, a proof
is provided in Section~\ref{sub:percolation_on_IPS}. It shows that IPS has some kind of
critical behavior in interval $(0,\bar{u})$. Our work investigates if
the model also presents a critical geometry.
Notice that it is not clear that the existence of such `vacant crossings' is
well-defined since the fact that small sticks are dense in the plane may raise
measurability issues. Also, some subtle topological issues have to be
considered, see Remark~\ref{rem:topological_issue}.  In
Section~\ref{sec:exploration_paths_IPS_ellipses} we give a precise definition
to such crossings via a natural family of exploration paths $\{\gamma_r^0\}_{r
> 0}$, see Figure~\ref{fig:example_exp_path}.
\begin{figure}
\centering
\begin{tikzpicture}[scale = 2.5,
   occupied/.style={pattern=north east lines,pattern color=gray!80}
]

\draw           (-2,-1  ) rectangle (2, 1);
\fill[occupied] (-2,-1.3) rectangle (2,-1);
\draw[dashed]   (-2.3,-1) rectangle (-2,1);

%text
 \node at (0,-1.15) {Covered};
 \node[rotate = 90] at (-2.1,0) {Vacant};

%sticks intersecting box that form the exploration path
\clip (-2,-1  ) rectangle (2, 1);

\draw[name path = s1, rotate around={ 50:(-1.1, -1.1)}] (-1.1,-1.1) circle (1 and 0);
\draw[name path = s2, rotate around={-10:(-1.1, -0.5)}] (-1.1,-0.5) circle (1 and 0);
\draw[name path = s3, rotate around={ 15:(-1.7, -0.5)}] (-1.7,-0.5) circle (.7 and 0);
\draw[name path = s4, rotate around={ 90:(-1.1, -0.1)}] (-1.1,-0.1) circle (.3 and 0);
\draw[name path = s5, rotate around={  5:(   0,  0.2)}] (   0, 0.2) circle (1.2 and 0);
\draw[name path = s6, rotate around={100:( 1.2,  0.2)}] ( 1.2, 0.2) circle (.6 and 0);
\draw[name path = s7, rotate around={  3:( 1.6, -0.2)}] ( 1.6,-0.2) circle (.6 and 0);

%exploration path
\draw[line width=.8mm] (-2,-1) -- (-1,-1) -- ++(50:.53) --
    ++(170:.83) -- (-2,-.6); %part 1
\draw[line width=.8mm] (-2,-.33) -- ++(-10:.54) -- ++(15:.4) --
    ++(90:.42) -- ++(185:.12) -- ++(5:.12) -- ++(90:.1) --
    ++(-90:.1) -- ++(5:2.28) -- ++(100:.51) --
    ++(-80:1.03) -- (2,-.18); %part 2

%decorative sticks
\draw[rotate around={110:(1,-1)}] ( 1,-1) circle (1 and 0);
\draw[rotate around={60:(.7,-1)}] (.7,-1) circle (.6 and 0);
\draw[rotate around={-6:(-1.5,.6)}] (-1.5,.6) circle (.4 and 0);
\end{tikzpicture}
\caption{Depiction of a random curve from family $\{\gamma_r^0\}_{r > 0}$. For
    $r>0$ fixed, only finitely many sticks from $\cE$ intersect the box.
    The exploration path $\gamma_r^0$ starts at the bottom-left corner and follows the
    vacant/covered interface without `crossing' sticks, until reaching either the
    right or top sides. As $r \to 0$ more sticks are added to the picture and
    the exploration path is updated. See
    Section~\ref{sec:exploration_paths_IPS_ellipses} for a precise
    definition.}
\label{fig:example_exp_path}
\end{figure}
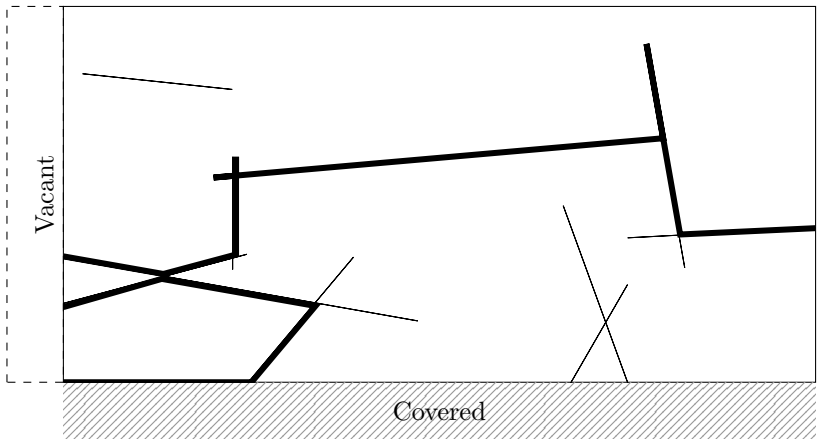

Exploration processes are quite common in the literature and have great
importance in percolation theory. They appear in Russo Seymour Welsh
theory~\cite{russo1978note, seymour1978percolation}, in the study of
noise sensitivity~\cite{garban2014noise} and also on the proof that the
interface of critical percolation on the triangular lattice converges to
$\SLE_6$~\cite{camia2007critical}.

The exploration paths $\{\gamma^0_r\}_{r > 0}$ that we use reveal the
interface between covered and vacant regions of approximations of
IPS that get increasingly better as $r \to 0$. Our main contribution to
understanding exploration paths $\{\gamma^0_r\}$ is
\begin{teo}
\label{teo:tight_regularity_exploration_paths}
Let $u \in (0, \bar{u})$ and consider an IPS $\xi$ of intensity $u$ and
a box $B \subset \RR^2$. The family of exploration curves
$\{\gamma^0_r\}_{r > 0}$ in box $B$ is tight and there is $c(u) > 0$
such that any subsequential limit law in the space of curves is
supported on Holder continuous curves with upper-box dimension and
Hausdorff dimension in the interval $(1, 2-c(u))$.
\end{teo}
For a more precise statement, see
Theorem~\ref{teo:tight_regularity_exploration_paths_2}. All properties in
Theorem~\ref{teo:tight_regularity_exploration_paths} are consequences of
previous results by Aizenman and Burchard~\cite{aizenman1999holder}.
The proof of the lower bound on Hausdorff dimension is based on
\textit{Property ($\varnothing$)}, a concept introduced by Basdevant, Blanc,
Curien and Singh~\cite{basdevant} which is inspired in Theorem 1.3
of~\cite[]{aizenman1999holder}.
For the upper bound $2-c(u)$, we prove that our family of exploration paths 
satisfy a property called \textit{Hypothesis H1} in \cite{aizenman1999holder},
which we abbreviate to \textbf{H1}. Intuitively, it quantifies how erratic
the family of curves $\{\gamma^{0}_r\}$ can be by estimating the probability
that the curve $\gamma^{0}_r$ traverses an annulus many times. 
\begin{teo}[\textbf{H1} for exploration paths]
\label{teo:hypothesis_H1_IPS}
Let $u \in (0, \bar{u})$.
For all $k \in \NN$ and $0 < l_1 < l_2$ we have uniformly in
$r > 0$ and $z \in \RR^2$ that
\begin{equation}
\label{eq:hypothesis_H1_IPS}%
\PP\biggl( 
    \begin{array}{l}
    \text{Annulus $\{w; l_1 < |w-z| \le l_2\}$ is traversed} \\
	\text{by $k$ separate segments of curve $\gamma^{0}_r$}
    \end{array}
\biggr)
    \le K_k \left(\frac{l_1}{l_2}\right)^{c(u)k}
\end{equation}
for positive constants $K_k(u) < \infty$ and $c(u) > 0$.
\end{teo}
Reference~\cite{aizenman1999holder} relates \textbf{H1} to tightness of
the family of curves and regularity of their possible limits.
Given that Theorem~\ref{teo:hypothesis_H1_IPS} holds, the upper bound in
Theorem~\ref{teo:tight_regularity_exploration_paths} is a straightforward
consequence of Theorems~1.1 and 1.2 from~\cite{aizenman1999holder}.

\subsection{Previous results}
\label{sub:previous_results}

Nacu and Werner~\cite{nacu2011random} studied a class of Poissonian translation and
scale-invariant models in the plane that include IPS. In order to motivate and
contextualize our results, we make a brief summary of known properties of IPS.

Let $B(l)$ be an (euclidean) ball of radius $l$. IPS is a \textit{thin soup} as defined in~\cite{nacu2011random},
meaning that
\begin{equation*}
\mu_2(s; R \geq r, E_0(s) \cap B(1) \neq \varnothing) < \infty
\quad \text{for every $r>0$.}
\end{equation*}
There, they prove that for any fixed annulus in $\RR^2$ only a
finite number of sticks in $\xi$ intersect both its internal and external
boundaries. These results for IPS can be seen as particular cases of our
Proposition~\ref{prop:intersection_computations} and
Lemma~\ref{lema:crossing_annulus}, that consider HPS models in general.

About percolation of the covered set, they prove that there is a critical
parameter $u_c \in (0, \infty)$ that separates two distinct regimes:
\begin{description}
\item[Subcritical phase.] For $0 < u < u_c$ the soup is composed of
    bounded covered clusters;
\item[Supercritical phase.] For $u > u_c$ there is a unique and dense
    unbounded covered cluster.
\end{description}
Actually, their argument is phrased for random curves with positive inner area
instead of IPS, but essentially the same argument works for IPS, see
Lemma~\ref{lema:NW_properties}. They also introduce the
\textit{Subcriticality Assumption 2}, which we denote \textbf{SA2} for short.
Assumption \textbf{SA2} means that with positive probability, there
exists a (random) closed loop $l$ in the plane that surrounds the origin and
does not `cross' any sticks of the full soup.

Nacu and Werner also proved that IPS with small densities satisfy \textbf{SA2},
by claiming that the same coupling with fractal percolation present
in~\cite{sheffield2012conformal} works. This is essentially the content of
Proposition~\ref{prop:vacant_crossing_boxes_IPS}, since an application of FKG
inequality implies \textbf{SA2} holds for $u \in (0, \bar{u})$ and hence
$\bar{u} \le u_c$. Another easy comparison is $u_c \le \ustick$, where
$\ustick$ is the critical point of a Poisson stick soup with sticks of radius
1.  In~\cite{nacu2011random} it is argued that $\bar{u} = u_c$ should hold for
most cases of interest, including the Brownian loop soup. However, this is not
clear for IPS since it has a different notion of `crossing' loops of the soup,
see Definition~\ref{defi:cross_segment_by_curve}.

The main result from~\cite{nacu2011random} is about the Hausdorff dimension of
a set they call \textit{carpet}, under \textbf{SA2}. For any simply connected domain
$D \subset \RR^2$ define the random soup in $D$, denoted by $\Gamma_D$, as the
random sticks of $\xi$ that are contained in $D$. For $D \neq \RR^2$, the
\textit{carpet} $G$ is defined as the set of points $z \in D$ such that for any
neighborhood of $z$ there is a path connecting it to $\partial D$ that does not
`cross' any stick of $\Gamma_D$.

We emphasize that when analyzing IPS in region $D$, Nacu and Werner
consider only sticks in $\Gamma_D$. This is a very important difference to our
analysis of exploration paths, since we cannot disregard sticks intersecting
the boundary of our regions, as discussed in Section~\ref{sub:discussion}.
Under \textbf{SA2}, Nacu and Werner were able to estimate the probability of
a vacant 1-arm event. Let $D(l_1,l_2) := \{z \in \RR^2; l_1 < |z| \le l_2\}$
and consider event $A_{\epsilon}$ in which there is a path joining the
inner and outer boundaries of $D(\epsilon, 1)$ that do not `cross' any
sticks of $\Gamma_{D(\epsilon, 1)}$. Their Corollary~8 proves there are
constants $\bar{\eta}(u)$ and $k'(u)$ such that
\begin{equation}
\label{eq:nacu_poly_bound}
\epsilon^{\bar{\eta}} \le \PP(A_\epsilon) \le k'\epsilon^{\bar{\eta}},
\end{equation}
a power law for $\PP(A_\epsilon)$. This vacant 1-arm estimate is
related to the Hausdorff dimension of the carpet. If $D$ is a bounded
non-empty open domain then its carpet $G$ has Hausdorff dimension
$2 - \bar{\eta}$ almost surely. Finally, they study how
$\dim_{\cH}(G)$ behaves as $u$ decreases to zero, obtaining that
$\dim_{\cH}(G) = 2 - o(u)$.

\begin{remark}
\label{remark:Hausdorff_corollary}
It is worth mentioning that exploration paths of a box are contained in its
carpet, so the conclusion in
Theorem~\ref{teo:tight_regularity_exploration_paths} that limiting subsequential
curves have Hausdorff dimension smaller than $2-c(u)$ can be seen as a corollary
of Nacu and Werner. However, tightness of the sequence of curves and estimates
on their Holder continuity are not a consequence of~\cite{nacu2011random}, see
Section~\ref{sub:discussion}.
\end{remark}

\subsection{Discussion and Idea of Proofs}
\label{sub:discussion}

Our main results, Theorems~\ref{teo:tight_regularity_exploration_paths}
and~\ref{teo:hypothesis_H1_IPS}, have a similar flavor
to the ones just described. Denote by $C_\epsilon$ the event in which
there is a finite sequence of sticks of the full soup connecting the
internal and external boundaries of annulus $D(\epsilon,1)$, or in
other words, in $C_\epsilon$ we have a covered 1-arm. The estimate \textbf{H1} in
Theorem~\ref{teo:hypothesis_H1_IPS} is based on a polynomial decay estimate
for $\PP(C_\epsilon)$: there are constants $K(u)$ and $\eta(u)$ such that
\begin{equation}
\label{eq:C_eps_poly_bound}
\PP(C_\epsilon) \le K\epsilon^{\eta},
\end{equation}
see Proposition~\ref{prop:up_bound_arm_IPS}.
We emphasize that both events $A_{\epsilon}$ and $C_{\epsilon}$ search for
paths that cross the annulus $D(\epsilon,1)$, but are different in nature.
The event $A_{\epsilon}$ concerns paths that may even intersect sticks from
$\Gamma_{D(\epsilon,1)}$, but do not `cross' them, while $C_{\epsilon}$
concerns paths entirely contained in the covered set.

While the lower bound in~\eqref{eq:nacu_poly_bound} can be adapted for
$C_{\epsilon}$ (see Lemma~\ref{lema:lbound_power_law_C_epsilon}), the estimates
in~\eqref{eq:C_eps_poly_bound} and~\eqref{eq:hypothesis_H1_IPS} rely on new
ideas to overcome two extra difficulties: the lack of independence and the
application of BK inequality.

\medskip
\noindent
\textbf{Lack of independence.}
Events $A_{\epsilon}$ look for paths that do not cross sticks in
$\Gamma_{D(\epsilon,1)}$, disregarding boundary effects of sticks intersecting 
$\partial D(\epsilon,1)$. The upper bound in~\eqref{eq:nacu_poly_bound} follows
from independence of finding paths that do not cross sticks in
$D(\epsilon\epsilon', \epsilon)$ and in $D(\epsilon, 1)$.
The same argument does not apply for covered 1-arms, since paths of
sticks in the two annuli above must share sticks.

To overcome this lack of independence we decompose annulus $D(1, 2^{m})$ into
annuli of the form $D(2^{j-1}, 2^{j})$. If there is a 1-arm covered crossing,
then none of these annuli can contain a circuit around the origin that do not
cross sticks of the soup, and such events have probability bounded away from 1.
These events are also not independent, but independence can be enforced by
considering an appropriate (random) subset of indices $j$ obtained by exploring
these annuli from the outermost to the innermost.
Estimate~\eqref{eq:C_eps_poly_bound} follows from the fact that both events
`there are many independent annuli without a circuit around the origin' and
`there are too few attempts' have small probability. This is the content of
Proposition~\ref{prop:up_bound_arm_IPS} in
Section~\ref{sub:1_arm_events_for_sipss}.

\begin{remark}
We are not able to fully recover a power law for $\PP(C_{\epsilon})$ with
matching exponents, as it holds for $\PP(A_{\epsilon})$
in~\eqref{eq:nacu_poly_bound}. However, comparing IPS with other percolation
models in general, it is quite plausible that the covered and vacant sets of IPS
have different behaviors since the model does not present a clear symmetry
between covered and vacant sets. As a close example of this lack of
symmetry in continuum percolation, we can cite
reference~\cite{ahlberg2018sharpness}. Their Corollary~4.4 presents bounds
for the probability of covered and vacant arm events in Boolean models on
the plane that are unavoidably different. For some radius distributions,
the probability of having a vacant arm decays polynomially but the decay
for covered arms is slower.
\end{remark}

\medskip
\noindent
\textbf{BK inequality.}
The covered 1-arm estimate from Proposition~\ref{prop:up_bound_arm_IPS}
is an important tool to prove that our family of exploration curves
satisfies \textbf{H1}, but is not enough. Since our exploration
paths `follow' sticks, it is natural to expect that proving \textbf{H1} could
follow these arguments:
\begin{equation}\begin{array}{c}
\label{disp:exploration_arms}    
\text{ if one has many separate segments of curve $\gamma^{0}_r$ traversing the
    same annulus then}\\
\text{there are many covered arms crossing the annulus that use disjoint sets of sticks}
\end{array}\end{equation}
and then apply BK inequality to obtain \textbf{H1}, a standard inequality for
bounding the disjoint occurrence of events (see Section~\ref{sub:bk_inequality}).
Unfortunately, the possibility that different segments of the curve
share a same stick makes the reasoning
behind~\eqref{disp:exploration_arms} more convoluted. Again, this does not play
a role in estimate~\eqref{eq:nacu_poly_bound} from~\cite{nacu2011random}, since
$\Gamma_{D(\epsilon \epsilon', \epsilon)}$ and $\Gamma_{D(\epsilon, 1)}$ do not
share sticks.
We are indeed able to conclude that~\eqref{disp:exploration_arms} holds, but to
do so we rely on topological properties of our exploration paths (see
Proposition~\ref{lema:topological_properties}) and the fact that the number of
sticks of the soup that intersect $\{z; |z| = 1\}$ in two different points is
finite a.s., see Proposition~\ref{prop:sticks_cap_boundary_ball}.

\medskip
\noindent
\textbf{Open problems.}
This work allows for various extensions and questions. We expect that ellipses model
converges to IPS on Smirnov Schramm topology. Another interesting result
would be to prove that the probability of having a covered 1-arm event,
$\PP(C_{\epsilon})$, has a power law decay with some explicit exponent, and if
this is the case, how the exponents of vacant and covered 1-arms are related.
It also could be interesting to ensure that family $\{\gamma_r^r\}_{0 < r \le 1}$
satisfies \textbf{H1} (see
Section~\ref{sec:exploration_paths_IPS_ellipses}). Possibly
the same argument used for $\{\gamma_r^0\}_{0 < r \le 1}$ could be adapted
to this other family of curves. However, we are still not able to fill in
the details. One could reasonably expect that in the limit curves
$\gamma^r_r$ and $\gamma^0_r$ get arbitrarily close almost surely.

\medskip
\noindent
\textbf{Structure of the paper.}
In Section~\ref{sec:scaling_ellipses_model} we show that HPS measures
$\mu_{\alpha}$ given in~\eqref{eq:shpss_measure} are indeed what we
expect when considering weak convergence of the PPP's defining ellipses
model. Section~\ref{sec:some_intersection_events} collects estimates on the
probability of HPS intersecting balls and segments. These estimates
are useful on many computations and also quantify the decay of
correlations of a HPS model. In
Section~\ref{sec:exploration_paths_IPS_ellipses} we focus on IPS model.
The definition of 
our exploration paths $\gamma^{0}_r$ together with some of its properties is in
Section~\ref{sub:exploration_paths_and_their_properties}. Finally,
Section~\ref{sec:Holder_regularity_exploration_paths} contains the proof of
Theorems~\ref{teo:tight_regularity_exploration_paths}
and~\ref{teo:hypothesis_H1_IPS}.

%\medskip
%\noindent
%\textbf{Acknowledgments.} Most of this work was developed in DU Phd Thesis at
%IMPA. The research of D.U. had financial support by FAPERJ, grant 202.231/2015,
%and FAPESP, grant 2020/05555-4. A.T. was supported by grants “Projeto Universal” (406250/2016-
%2) and “Produtividade em Pesquisa” (304437/2018-2) from CNPq and “Jovem
%Cientista do Nosso Estado”, (202.716/2018) from FAPERJ.

\section{Scaling ellipses model}
\label{sec:scaling_ellipses_model}
% previously sec:poisson_stick_soup

In order to consider homotheties of ellipses models it is useful to see all
these models as being random subsets based on Poisson Point Processes (PPPs) on
$S := \RR^2 \times \RR^+ \times (-\pi/2, \pi/2]$, usually denoted by
$\xi$. From $\xi$ we can build collections of objects (ellipses or sticks) on
the plane.

Fix $l \geq 0$ and for every point $s = (z, R, V)\in \xi$ with $R \geq l$ 
we define $E_l(s)$ as the ellipse with minor axis of size $l$ 
and coordinates $z, R, V$ being its center, major axis size
and direction, respectively. By this, a point process $\xi$ on
$S \cap \{R \geq l\}$ can be used to build the subset of the plane
$\cE_l(\xi) = \cup_{s \in \xi} E_l(s)$, which we call the \textit{covered} set.
Its complementary set $\cV_l(\xi)$ is called the \textit{vacant} set.
Notice that in the degenerate case $l = 0$ we have a collection of sticks.

In a previous work~\cite{teixeira_ungaretti2017ellipses} we use this setup with
$l=1$ to define the \textit{ellipses model}, a collection of heavy-tailed
random ellipses with size of their minor axis always equal to one.
Consider a PPP $\xi$ on $S$ with intensity measure $u\lambda \otimes
\rho \otimes \nu$ where $\lambda$ is the Lebesgue measure on $\RR^2$, $u>0$ is
a parameter that controls the density of ellipses, $\nu$ is the uniform
probability measure on $(-\pi/2, \pi/2]$ and $\rho$ is supported on $[1,
\infty)$ and satisfies $\rho[r, \infty) = L(r) r^{-\alpha}$ for some
\textit{slowly varying} function $L$, i.e., a function
$L: \RR^+ \to \RR^+$ such that for every $a > 0$ it holds $L(ax)/L(x) \to 1$
as $x \to \infty$. Parameter $\alpha > 0$ controls the tail decay of $\rho$.

The homothety of ratio $l > 0$ preserves directions and contracts space by a factor
$l^{2}$. It induces a transformation on $S$ defined as $\cH_l: S \to S$ such that
$(z, R, V) \mapsto (lz, lR, V)$ and we also denote
$\cH_l: \xi = \sum_i \delta_{s_i} \mapsto \sum_i \delta_{\cH_{l}(s_i)}$.
Thus, if $\xi$ is a PPP on $S$ with intensity measure $u\lambda \otimes
\rho \otimes \nu$ then $\cH_l\xi$ is also a PPP, but with intensity measure
\begin{equation*}
ul^{-2} \lambda \otimes \rho_l \otimes \nu,
    \quad \text{where $\rho_l[r,\infty) := \rho[\tfrac{r}{l},\infty)$},
\end{equation*}
and we want to tune $u(l)$ so that $\cH_l\xi$ converges to a
non-degenerate PPP on $S$ as $l \to 0$.
\begin{lema}[PPP convergence]
\label{lema:ppp_convergence}
Let distribution $\rho$ satisfy $\rho[r, \infty) = L(r) r^{-\alpha}$ for some
slowly varying function $L$ and for a fixed constant $\beta > 0$ define
$u(l) = \beta \cdot L(1/l)^{-1} \cdot l^{2-\alpha}$ and
$\xi_l = \PPP(u \lambda \otimes \rho \otimes \nu)$. The $\PPP$ on $S$ given by
$\cH_l\xi_l$ converges weakly to
$\PPP(\beta \mu_{\alpha})$ as $l \to 0$,
where $\mu_{\alpha}$ is the measure
\begin{equation*}
    \mu_{\alpha} := \lambda \otimes \phi_{\alpha}(\mathrm{d}x) \otimes \nu,
\end{equation*}
with $\phi_{\alpha}(\mathrm{d}x) := \alpha x^{-(1+\alpha)} \,\mathrm{d}x$.
\end{lema}

\begin{proof}
See~\cite[Section 3.1.1]{ungaretti2017phdthesis}.
\end{proof}

Measure $\mu_{\alpha}$ has a scaling property that resembles
homogeneity, as shown by
\begin{prop}
\label{prop:homogeneity_relation_PSS}
For any $\alpha, c > 0$ we have the homogeneity relation
$\mu_\alpha(\cH_c( \cdot )) = c^{2 - \alpha} \mu_\alpha$. In particular,
measure $\mu_2$ is scale-invariant.
\end{prop}

\begin{proof}
For any rectangular event $I$ of the form
$I = K \times (a,b) \times J$ it holds
\begin{align*}
\mu_\alpha\bigl(\cH_c(I)\bigr)
    &= \mu_\alpha\bigl(cK \times (ca,cb) \times J\bigr)
    =  c^{2} \lambda(K) \cdot
        \int_{ac}^{bc} \alpha x^{-(1 + \alpha)} \, \text{d}x \, \cdot \nu(J) \\
&= c^{2} \lambda(K)
    \bigl((ac)^{-\alpha} - (bc)^{-\alpha}\bigr)
    \nu(J)
= c^{2 - \alpha} \lambda(K) \bigl(a^{-\alpha} - b^{-\alpha}\bigr) \nu(J) \\
  &= c^{2 - \alpha} \mu_\alpha (K \times (a,b) \times J).
\end{align*}
Since equality holds for sufficiently many sets, the measures
$\mu_\alpha (\cH_c(\cdot))$ and $c^{2 - \alpha}\mu_\alpha$ are equal.
\end{proof}

%As mentioned before, if we take $\alpha = 2$ in
%Proposition~\ref{prop:homogeneity_relation_PSS} we conclude that IPS
%model is scale-invariant.

\begin{remark}
\label{rem:smirnov_schramm}
As mentioned in the introduction, one could study convergence in the
Smirnov-Schramm topology, defined on the space of quads
(topological quadrilaterals), denoted by $(\cQ, \cT)$, for studying scaling
limits of planar percolation models. The idea behind this topology is that a
percolation model can be associated to a probability distribution on $\cQ$ that
codifies all quads that have been crossed by open paths. Under some uniform
Russo Seymour Welsh estimates on the sequence of percolation models, it holds
that their respective probability distributions on $\cQ$ are precompact with
respect to weak convergence. If we consider a sequence of ellipses models with
$\alpha = 2$, Theorem~1.3 from~\cite{teixeira_ungaretti2017ellipses}
asserts that
\begin{equation}\begin{array}{c}
\label{disp:box_cross}
\text{the probability of crossing boxes of fixed proportion with a}\\
\text{left-right vacant path is bounded away from $0$ and $1$,}
\end{array}\end{equation}
implying that a subsequential limiting distribution might be non-trivial.
We have not followed this line of work.
\end{remark}

\section{Some intersection events}
\label{sec:some_intersection_events}

In this section we study the behavior of small sticks and large sticks on a
HPS. Here, homogeneity is a very useful tool. A good example is applying
homogeneity to study percolation for the covered and the vacant sets, as shown
by
\begin{prop}
\label{prop:percolation_when_alpha_neq_2}
For $\alpha \neq 2$ and any intensity $u>0$, a HPS satisfies
$\PP(\cE \ \text{percolates}) = 1$ and $\PP(\cV \ \text{percolates}) = 0$.
\end{prop}

\begin{proof}
One important observation is that whenever we restrict the radius of the sticks
of a HPS to a finite interval, we can easily compare the soup obtained
with a Poisson stick soup (PSS) of sticks with fixed length, a model that
is already quite well-understood (see \cite{meester1996continuum}). Indeed,
notice that for any $r > 0$
\begin{equation*}
u \phi_\alpha ((r/2, r]) \nu((- \tfrac{\pi}{2}, \tfrac{\pi}{2}])
    = u (2^{\alpha} - 1) r^{-\alpha}
\end{equation*}
is the intensity of the process obtained when we only consider sticks with
radius in the interval $(r/2,r]$. In this restricted range of radii, the covered
set of HPS contains the covered set of a PSS with sticks of radius $r/2$ and intensity
$u (2^{\alpha} - 1) r^{-\alpha}$ and is contained in the covered set of a PSS of radius $r$
and same intensity. Moreover, if we denote by $\ustick(l)$ the critical parameter
for PSS of radius $l$, we must have by scaling that $\ustick(l) =
\ustick(1)l^{-2}$. Indeed, just as in Section~\ref{sec:scaling_ellipses_model}
we can use the homothety of ratio $l > 0$ to couple a PSS of radius 1 and
intensity $u$ to a PSS of radius $l$ and intensity $ul^{-2}$.
Notice that if $\alpha > 2$ then
\begin{equation}
\label{eq:density_ratio_for_dominating_PSS}
u (2^{\alpha} - 1) r^{-\alpha} \gg 4\ustick(1) r^{-2} = \ustick(r/2)
    \quad \text{when $r \to 0$.}
\end{equation}
This means that for $\alpha > 2$ small sticks
dominate fixed length PSSs with arbitrarily high intensities, and thus
$\cE$ percolates but $\cV$ does not. Analogously, if we fix $\alpha < 2$ we
have $u (2^{\alpha} - 1) r^{-\alpha} \gg \ustick(r)$ when $r \to \infty$
and the same conclusion applies, changing small sticks for large ones. We
emphasize that this reasoning is valid for any fixed $u > 0$.
\end{proof}

We also list some computations and estimates for the probability of
relevant events for HPS. Some of them are obtained through the same kind of
analysis done in~\cite{teixeira_ungaretti2017ellipses}. Another useful tool is
reference~\cite{parker1976some}, which computes the expectation of some useful
quantities for general random processes of line segments on the plane.

We begin by studying some events that
depend only on sticks with radius within some finite range.
Let $t > r > 0$ and consider the soup $\xi^{r, t}$,
whose intensity measure is given by
\begin{equation*}
u\mu_{\alpha}^{r,t}
    = u\bigl(r^{-\alpha} - t^{-\alpha}\bigr) \lambda \otimes
    \bigl(r^{-\alpha} - t^{-\alpha}\bigr)^{-1}
    \I\{r \le x < t\} \cdot \phi_{\alpha}(\mathrm{d} x) \otimes
    \nu
\end{equation*}

With this representation, the marginal measure on $\RR^{+}$ is a
probability distribution and the model fits into the context of
\cite{parker1976some}. For any convex Borel set $A \subset \RR^{2}$
the expected number $N(A)$ of segments intersecting $A$ is given
by~\cite[Equation~(7)]{parker1976some},
\begin{equation}
\label{eq:parker_cowan_shpss_0}
\EE_{\alpha}^{r, t}[N(A)]
    = u\bigl(r^{-\alpha} - t^{-\alpha}\bigr) \lambda(A) +
    u\bigl(r^{-\alpha} - t^{-\alpha}\bigr)\pi^{-1} \EE_{\alpha}^{r, t}[2R]
    \cdot \per(A),
\end{equation}
where $\smash{\EE_{\alpha}^{r, t}[2R]}$ is the expected length of a
stick and $\per(A)$ is the perimeter of $A$. Computing the expected value of
$R$, one gets
\begin{equation*}
\EE_{\alpha}^{r, t}[R] 
    = \bigl(r^{-\alpha} - t^{-\alpha}\bigr)^{-1}
    \int_{r}^{t} x \cdot \alpha x^{-(1+\alpha)} \,\mathrm{d} x.
\end{equation*}
Noticing that the left hand side of~\eqref{eq:parker_cowan_shpss_0} can be
described as
\begin{equation*}
\EE_{\alpha}^{r, t}[N(A)]
    = \mu_{\alpha}^{r, t}(E_0(s) \cap A \neq \varnothing)
    = \mu_{\alpha}(E_0(s) \cap A \neq \varnothing, r \le R < t),
\end{equation*}

Equation~\eqref{eq:parker_cowan_shpss_0} can be rewritten as
\begin{equation}
\label{eq:parker_cowan_shpss}
\EE_{\alpha}^{r, t}[N(A)]
    = u\bigl(r^{-\alpha} - t^{-\alpha}\bigr) \lambda(A) +
    \frac{2 \alpha u}{\pi} \cdot
    \Bigl(\int_{r}^{t} x^{-\alpha} \,\mathrm{d} x\Bigr)
    \cdot \per(A).
\end{equation}
Even without the aid of
Equation~\eqref{eq:parker_cowan_shpss} it is clear that any fixed Borel
set $A$ with $\lambda(A) > 0$ has infinitely many small sticks
centered on it, since $\phi_\alpha(0,r) = \infty$ for any $r, \alpha >0$.
Using~\eqref{eq:parker_cowan_shpss} we can
also calculate the $\mu_{\alpha}$ measure of sticks
intersecting some fixed sets.
\begin{prop}
\label{prop:intersection_computations}
Fix $a, r > 0$.
Denote by $L(a) \subset \RR^2$ a segment of length $a$. We have
\begin{align*}
\mu_{\alpha}
\bigl( E_0(s) \cap L(a) \neq \varnothing, R \in [r, \infty) \bigr)
    &= \infty
    &\quad \text{for $\alpha \le 1$},
    \\
\mu_{\alpha}
\bigl( E_0(s) \cap L(a) \neq \varnothing, R \in [r, \infty) \bigr)
    &= \tfrac{4}{\pi} \cdot \tfrac{\alpha}{\alpha - 1} \cdot
    a r^{1 - \alpha}
    &\quad \text{for $\alpha > 1$},
    \\
\mu_{\alpha}
\bigl( E_0(s) \cap L(a) \neq \varnothing, R \in (0, r) \bigr)
    &= \infty
    &\quad \text{for $\alpha \geq 1$}, \\
\mu_{\alpha}
\bigl( E_0(s) \cap L(a) \neq \varnothing, R \in (0, r) \bigr)
    &= \tfrac{4}{\pi} \cdot \tfrac{\alpha }{1 - \alpha} \cdot a r^{1-\alpha}
    &\quad \text{for $\alpha < 1$}.
\end{align*}
For a ball $B(a) \subset \RR^2$ of radius $a$ we have
\begin{align*}
\mu_{\alpha}
\bigl( E_0(s) \cap B(a) \neq \varnothing, R \in (0, r) \bigr)
    &= \infty
    &\quad \text{for $\alpha > 0$}, \\
\mu_{\alpha}
\bigl( E_0(s) \cap B(a) \neq \varnothing, R \in [r, \infty) \bigr)
    &= \infty
    &\quad \text{for $\alpha \le 1$},
    \\
\mu_{\alpha}
\bigl( E_0(s) \cap B(a) \neq \varnothing, R \in [r, \infty) \bigr)
    &= \pi a^2 r^{-\alpha} +
    4 \cdot \tfrac{\alpha}{\alpha - 1} \cdot a r^{1 - \alpha}
    &\quad \text{for $\alpha > 1$}.
\end{align*}
\end{prop}

\begin{proof}
Use~\eqref{eq:parker_cowan_shpss}  and take limits when
$r \downarrow 0$ or $t \uparrow \infty$. 
\end{proof}

Curves that have some part that is close to being linear
should in principle present the same behavior as $L(a)$, but we
are not interested in investigating the behavior of general sets.
Instead, we analyze in depth the $\mu_{\alpha}$ measure of sticks
intersecting the boundary of a circle, which will be important
when studying arm events for IPS. In particular, we are interested in the
$\mu_\alpha$ measure of sticks that intersect a circle twice.
\begin{prop}
\label{prop:sticks_cap_boundary_ball}
For any $r, l > 0$ it holds that
\begin{align}
\label{eq:sticks_cap_boundary_ball}
\mu_{\alpha}\bigl(
    E_0(s) \cap \partial B(l) \neq \varnothing, R \in (0, r)\bigr)
    &< \infty
\quad \text{if and only if $\alpha < 1$}, \\
\label{eq:sticks_cap_boundary_ball_twice}
\mu_{\alpha}\bigl(
    \# E_0(s) \cap \partial B(l) = 2\bigr)
    &< \infty
\quad \text{if and only if $\alpha \in (1, 3)$}.
\end{align}
Moreover, for $\alpha=2$ we have 
$\mu_2\bigl(\# E_0(s) \cap \partial B(l) = 2\bigr) = 2\pi$.
\end{prop}

\begin{proof}
We can suppose by homogeneity that $l = 1$. Define
\begin{equation*}
A_r := \bigl\{E_0(s) \cap \partial B(1) \neq \varnothing, R \in (0, r)\bigr\}.
\end{equation*}
We have for any
$\varepsilon \in (0, r)$ that
\begin{equation*}
\label{eq:small_sticks_cap_circle1}
\mu_\alpha\bigl(
    E_0(s) \cap \partial B(1) \neq \varnothing,
    R \in [\varepsilon, r)
    \bigr)
    \le \lambda(B(1 + r)) \cdot (\varepsilon^{-\alpha} - r^{-\alpha})
    < \infty,
\end{equation*}
since a stick of radius smaller than $r$ that intersects $B(1)$ must have
center inside $B(1+r)$.  Thus, either $\mu_\alpha(A_r)$ is finite for every
$r > 0$ or it is infinite for every $r > 0$.
Using~\eqref{eq:parker_cowan_shpss} we have that $N = N(B(1))$ satisfies
\begin{equation}
\label{eq:cap_ball}
\EE^{r, 1}_{\alpha} [N]
    = (r^{-\alpha} - 1) \pi +
    \frac{2 \alpha}{\pi} \cdot
    \Bigl(\int_{r}^{1} \! x^{-\alpha} \,\mathrm{d} x\Bigr)
    \cdot 2\pi
    = \int_{r}^{1} \! (\pi + 4x) \alpha x^{-(1+\alpha)} \,\mathrm{d} x.
\end{equation}
If we denote by $C$ the number of sticks that are entirely
contained within $B(1)$ and discount it from $N$ we
obtain precisely the sticks that intersect $\partial B(1)$.
For fixed $R, V$ let $A_{0}(R, V) \subset \RR^2$ denote the set
of centers $z$ for which the stick $(z, R, V)$ is
contained in $B(1)$. Region $A_{0}(R,V)$ is non-empty when $2R < 2$,
case in which its area is
\begin{equation*}
\lambda(A_{0}(R,V)) = 2 \cdot (\arccos R - \sin(2\arccos R)/2),
\end{equation*}
see Figure~\ref{fig:sticks_cap_ball} for a visual proof.
We have
\begin{figure}
\centering
\begin{tikzpicture}[scale=2]
    \draw (0,0) coordinate (o) circle (1);
    \coordinate (B)  at ( 1, 1.2);
    \coordinate (mB) at ( 1,-1.2);
    \coordinate (C)  at (-1, 1.2);
    \coordinate (mC) at (-1,-1.2);
    % region A(R,V)
    \draw[pattern=north west lines]
        (B) arc (0:-180:1) -- (C)
        -- (mC) arc (180:0:1) -- cycle;
    % distances
    \draw[|<->|] (mB) ++(.4, 0) -- ($(B) +(.4, 0) $) node[midway, right] {$2R$};
    \draw[|<->|] (-1.2, -1) -- (-1.2, 1) node[midway, left] {$2$};
    %labels
    \node at (0  , 1.3) {Case $2R \geq 2$};
    \node at (3.8, 1.3) {Case $2R < 2$};
    % vector V
    \draw[thick, ->, >=latex] (2, -1) -- (2, -.6) node[midway, right] {$V$};
\begin{scope}[shift={(3.8, 0)}]
    \draw (0,0) coordinate (o) circle (1);
    \coordinate (A)  at ({ cos(40)}, 0);
    \coordinate (mA) at ({-cos(40)}, 0);
    \coordinate (B)  at (  40:1);
    \coordinate (mB) at (- 40:1);
    \coordinate (C)  at ( 140:1);
    \coordinate (mC) at (-140:1);
    % angles
    \draw (mB) -- (B) -- (mC)
        pic["$\theta$", draw, angle eccentricity=1.7, angle radius=10] {angle=mC--B--mB};
    \draw (mC) -- (o) -- (mB)
        pic["$2\theta$", draw, angle eccentricity=1.9, angle radius=7] {angle=mC--o--mB};
    \draw (B) -- (mB) -- (mC)
        pic[draw, angle radius=8] {right angle=B--mB--mC}; %right angle
    % superposition area, A_0(R,V)
    \filldraw[black, fill=black!50, fill opacity=.4]
        (A) arc (40:140:1) arc (-140:-40:1) -- cycle;
    \draw
        (mB) -- (mC) arc (-140:-40:1) -- cycle;
    % region A_2(R,V)
    \draw[pattern=north west lines]
        (A) arc (-40:0:1) -- ++($(mB) - (B)$) arc (0:40:1) -- cycle;
    \draw[pattern=north west lines]
        (mA) arc (220:180:1) -- ++($(mB) - (B)$) arc (180:140:1) -- cycle;
    % distances
    \draw[|<->|] (mB) ++(.4, 0) -- ($(B) +(.4, 0) $) node[midway, right] {$2R$};
    \draw[|<->|] (-1.2, -1) -- (-1.2, 1) node[midway, left] {$2$};
\end{scope}
\end{tikzpicture}
\caption{Regions $A_{0}(R,V)$ (shaded) and $A_{2}(R,V)$ (hatched) represent
    centers $z$ for which $(z,R,V)$ intersects $\partial B(1)$ precisely $0$ and
    $2$ times, respectively. The area of $A_0(R,V)$
    is non-zero only for $2R < 2$, in which case it is given by
    $2 \cdot (\theta - \sin(2\theta)/2)$ (twice the area of a segment
    of the circle with angle $2\theta$, where $\theta = \arccos R$).
    The area of $A_{2}(R, V)$ when $2R \geq 2$ is $4R-\pi$,
    but when $2R < 2$ we must account for the superimposed
    area ($A_0(R,V)$). The area is then
    $4R-\pi + \bigl(2\theta - \sin(2\theta)\bigr)$.}
\label{fig:sticks_cap_ball}
\end{figure}
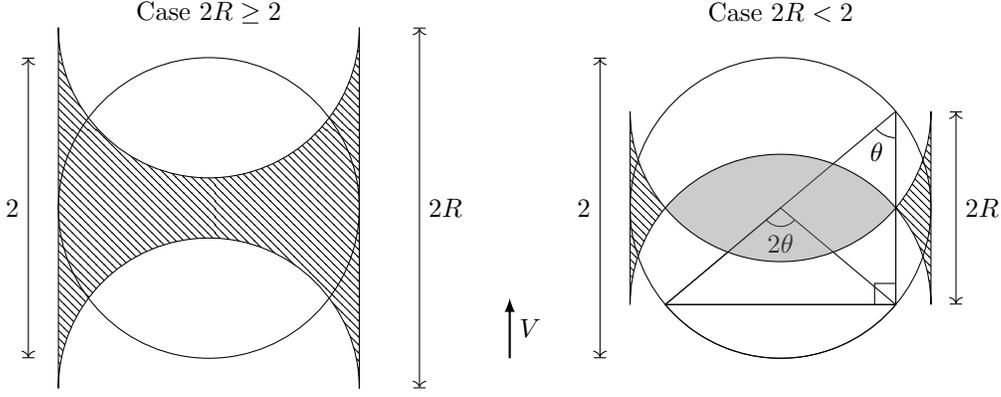
\begin{align*}
\EE_{\alpha}^{r,1}[C]
    = \mu_{\alpha}^{r,1}\bigl(E_0(s) \subset B(1)\bigr)
    &= \frac{1}{\pi} \int_{r}^{1} \int_{-\pi/2}^{\pi/2}
        \hspace{-4mm}\lambda(A_{0}(R, V))
        \,\mathrm{d} V \,\phi_{\alpha}(\mathrm{d} R) \\
    &= \int_{r}^{1}
        \Bigl(2\arccos x - 2x \sqrt{1 - x^{2}}\Bigr)
        \cdot \alpha x^{-(1+\alpha)}\,\mathrm{d} x,
\end{align*}
and thus we can write
\begin{equation*}
\EE_{\alpha}^{r,1}[N - C]
    = \int_{r}^{1}
    \Bigl(\pi + 4x - 2\arccos x + 2x \sqrt{1 - x^{2}}\Bigr)
    \cdot \alpha x^{-(1+\alpha)}\,\mathrm{d} x.
\end{equation*}
The integrand is non-negative and asymptotically
$\bigl(8x + O(x^{3})\bigr) \alpha x^{-(1+\alpha)}$
as $x \downarrow 0$, implying that
\begin{equation*}
\lim_{r\to 0} \EE_{\alpha}^{r,1}[N - C] = \infty
\quad \text{for $\alpha \geq 1$.}
\end{equation*}
Also, noticing that 
$\pi + 4x - 2\arccos x + 2x \sqrt{1 - x^{2}} \le 8x$ by concavity,
we obtain $\EE_{\alpha}^{0,1}[N - C] < \infty$ when $\alpha < 1$ and conclude
the proof of~\eqref{eq:sticks_cap_boundary_ball}.

To prove~\eqref{eq:sticks_cap_boundary_ball_twice} we adopt a
similar approach. Once again, homothety allows us to
consider only the case $l=1$.
Let $T$ denote the number of sticks that intersect
$\partial B(1)$ exactly twice, and for fixed $R, V$ denote
by $A_{2}(R,V)$ the region of centers $z$ of sticks $(z, R, V)$
which intersect $\partial B(1)$ twice when present. See
Figure~\ref{fig:sticks_cap_ball} for the area of $A_{2}(R, V)$.
We have
\begin{align*}
\EE_{\alpha}[T]
    &= \mu_{\alpha}\bigl(\# E_0(s) \cap \partial B(l) = 2\bigr)
    = \frac{1}{\pi} \int_{0}^{\infty} \int_{-\pi/2}^{\pi/2}
        \hspace{-4mm}\lambda(A_{2}(R, V))
        \,\mathrm{d} V \,\phi_{\alpha}(\mathrm{d} R) \\
    &= \int_{0}^{1} \hspace{-2mm}
        (4x-2\arcsin x - 2x \sqrt{1 - x^{2}})
        \phi_{\alpha}(\mathrm{d}x) + 
    \int_{1}^{\infty} \hspace{-2mm}
        (4x - \pi) \phi_{\alpha}(\mathrm{d}x) \\
    &=: (I) + (II).
\end{align*}
Both terms are positive. Notice that $(II)$ is infinite for $\alpha \le 1$,
while for $\alpha > 1$ we can write
\begin{equation*}
    (II)
    = 4\alpha \int_{1}^{\infty} x^{-\alpha} \,\mathrm{d} x -
    \pi \int_{1}^{\infty} \alpha x^{-(1+\alpha)} \,\mathrm{d} x
    = \frac{4\alpha}{\alpha-1} - \pi.
\end{equation*}
For the first term, we have that as $x \downarrow 0$ the integrand
is asymptotically $\tfrac{2}{3} x^3$. Since
$\int_{0}^{1} x^{2-\alpha} \,\mathrm{d} x$ is finite if and only if
$\alpha < 3$, the result in~\eqref{eq:sticks_cap_boundary_ball_twice}
follows. Finally, for $\alpha=2$ it is straightforward
to check that
\begin{equation*}
g(x)
    = 6 \cdot \frac{\smash{\sqrt{1-x^{2}}}}{x} + 
    2 \cdot \frac{\arcsin x}{x^2} + 4 \arcsin x - \frac{8}{x}
\end{equation*}
is an anti-derivative for the integrand in $(I)$, implying
\begin{equation*}
\mu_{2}\bigl(\# E_0(s) \cap \partial B(l) = 2\bigr)
    = g(1) - g(0) + \Bigl(\frac{4 \cdot 2}{2-1} - \pi\Bigr)
    = 2\pi. \qedhere
\end{equation*}
\end{proof}

\subsection{Long connections on HPS}
\label{sub:long_connections_on_shpss}

The estimates from this section make it possible to
quantify the probability of a single stick connecting far away regions.
From Proposition~\ref{prop:intersection_computations} we can estimate the
probability that one stick will intersect the outer and inner boundaries of an
annulus.
\begin{lema}
\label{lema:crossing_annulus}
% previously lema:up_bound_touching_two_annuli
Take $\alpha > 1$ and let
$
\Gamma_{12}
    := \{s;\ E_0(s) \cap \partial B(l_j) \neq \varnothing, j = 1,2\}
$
with $l_2 > l_1 > 0$ and $a := l_2/l_1$.
There is a constant $c_\alpha > 0$ such that
\begin{equation}
\label{eq:crossing_annulus_l2}
c_\alpha^{-1} l_1^{2 - \alpha}[ a^{-\alpha} + a^{1 - \alpha}] 
    \le \mu_\alpha(\Gamma_{12})
    \le c_\alpha l_1^{2 - \alpha} 
        [ (a - 1)^{-\alpha} + (a - 1)^{1 - \alpha}].
\end{equation}
In particular, if $a \geq 2$ then
\begin{equation}
\label{eq:crossing_annulus_simpler}
\mu_\alpha(\Gamma_{12})
    \le  c_\alpha l_1^{2 - \alpha} a^{1 - \alpha}.
\end{equation}
\end{lema}

\begin{proof}
By Proposition \ref{prop:intersection_computations} we have
\begin{align}
\mu_\alpha(\Gamma_{12})
    &\le \mu_\alpha
    \bigl(
    E_0(s) \cap B(l_1) \neq \varnothing,\, 2R \geq (l_2-l_1)
    \bigr)
    \nonumber\\
    &\le c_\alpha
    \bigl[
    l_1^2 (l_2 - l_1)^{-\alpha} + l_1 (l_2 - l_1)^{1 - \alpha}
    \bigr]
    \nonumber\\
    &=   c_\alpha l_1^{2 - \alpha}
    \bigl[
    (a - 1)^{- \alpha} + (a-1)^{1 - \alpha}
    \bigr],
\label{eq:up_bound_touching_two_annuli}
\end{align}
as we claimed for the upper bound. The lower bound follows the same reasoning,
beginning with the bound
\begin{equation*}
\mu_\alpha(\Gamma_{12})
    \geq \mu_\alpha
    \bigl(
    E_0(s) \cap B(l_1) \neq \varnothing,\, R \geq l_2
    \bigr).
\qedhere
\end{equation*}
\end{proof}
Using Lemma~\ref{lema:crossing_annulus} we are able to estimate
the decay of correlations for events depending on far away regions.
\begin{lema}[Decay of Correlations]
\label{lema:decay_of_correlations_HPS}
Take $\alpha > 1$. Let $K_1 = B(l_1)$ and $K_2 = {B(l_2)}^{\comp}$,
with $l_2 > l_1$.
Let $f_1$ and $f_2$ be real functions of $\xi$ such that
$|f_j| \le 1$ and $f_i$ depends only on sticks touching $K_i$.
Then, we have that
\begin{equation}
\label{eq:decay_of_correlations_HPS_0}
\bigl| \EE[f_1 f_2] - \EE[f_1]\EE[f_2] \bigr|
    \le 4u \mu_\alpha
    \bigl(s; E_0(s) \cap \partial B(l_j) \neq \varnothing, j = 1,2\bigr)
\end{equation}
In particular, if we take $l_2 = a l_1$ with $a \geq 2$ then
\begin{equation}
\label{eq:decay_of_correlations_HPS}
\bigl| \EE[f_1 f_2] - \EE[f_1]\EE[f_2] \bigr|
    \le  u c_\alpha l_1^{2-\alpha} a^{1 - \alpha}.
\end{equation}
\end{lema}

\begin{proof}
Inequality \eqref{eq:decay_of_correlations_HPS_0} is proved in Lema~6.1
of~\cite{teixeira_ungaretti2017ellipses}. The
bound on~\eqref{eq:decay_of_correlations_HPS} follows
from~\eqref{eq:decay_of_correlations_HPS_0} and
Lemma~\ref{lema:crossing_annulus}.
\end{proof}

\begin{remark}
\label{remark:sipss_ergodic}
Lemma \ref{lema:decay_of_correlations_HPS} gives a
measure of how fast a HPS mixes. Like ellipses model, any HPS
with $\alpha > 1$ is ergodic with respect to translations on $\RR^2$.
\end{remark}

Another important `large scale' event is the one in which we cross a box of
height $l$ and width $kl$ from side to side using only one stick.
Define the event
\begin{equation*}
LR_1(l;k)
    := \bigl\{\exists s \in \xi;
        \textup{$E_0(s)$ intersects both $\{0\} \times [0,l]$ and
        $\{kl\} \times [0,l]$}\bigr\}.
\end{equation*}

Replicating the same argument
from~\cite[Proposition~5.1]{teixeira_ungaretti2017ellipses}
we can prove the following estimates on the probability of $LR_1(l;k)$.
\begin{prop}
\label{prop:crossing_box_with_one_stick}
If $\alpha >1$ and $k, l > 0 $,
then there is $c(\alpha) > 0$ such that:
\begin{equation}
\label{eq:crossing_box_with_one_stick}
1 - e^{ - c^{-1} u (k \wedge k^{-\alpha}) l^{2 - \alpha} }
    \le \PP\bigl(LR_1(l;k)\bigr)
    \le 1 - e^{ - c u (k^{2-\alpha}\vee k^{-\alpha}) l^{2 - \alpha}}.
\end{equation}
\end{prop}

\section{Exploration paths in IPS}
\label{sec:exploration_paths_IPS_ellipses}

The HPS obtained when we fix $\alpha = 2$ is the most interesting HPS for
us, due to its scale invariance. As stated earlier, we refer to it as IPS.
In this section we study box-crossing events in IPS by introducing a
family of exploration paths that attests whether there is a vacant crossing
or not. For a IPS $\xi$ of intensity $u$ and any $r > 0$ we define
\begin{equation*}
    \xi_r := \xi \I\{R \geq r\},
\end{equation*}
the restriction of the soup to sticks with radius at least $r$. The idea is to
consider exploration paths $\gamma^0_r$ that do not cross sticks from $\xi_r$
and then make $r \downarrow 0$. Before that, we present some important
concepts about curves and define what we mean by `crossing' a stick of
the soup.

\subsection{Aizenman and Burchard}
\label{sub:the_space_of_curves}

We use the same definition of curves as in Aizenman and
Burchard~\cite{aizenman1999holder}. We refer
to~\cite{aizenman1999holder} for a more complete description.

\begin{defi}[Space of Curves]
We denote by $\cS_\Lambda$ the set of continuous functions from $[0,1]$
to a closed subset $\Lambda \subset \RR^2$ modulo reparametrizations.
We endow $\cS$ with the metric
\begin{equation*}
d(\cC_1, \cC_2)
    = \inf\limits_{\psi_1, \psi_2} \sup\limits_{t \in [0,1]}
        |f_1(\psi_1(t)) - f_2(\psi_2(t))|
\end{equation*}
where the infimum is over all continuous monotone bijections
$\psi_1$ and $\psi_2$ on $[0,1]$. In this case,
$(\cS_\Lambda, d)$ is a complete separable metric space,
see~\cite{aizenman1999holder}.
\end{defi}
There are some useful ways to quantify the regularity of a curve $\cC$.
In~\cite{aizenman1999holder} there are four different quantities used
for this purpose: Hausdorff dimension, upper box dimension, tortuosity and
H\"older exponent, denoted by $\dim_{\cH}(\cC), \dim_B(\cC), \tau(\cC)$ and
$\alpha(\cC)$, respectively. They satisfy the relation
\begin{equation*}
\dim_{\cH}(\cC) \le \dim_B(\cC) \le \tau(\cC)
    = \alpha(\cC)^{-1}
\end{equation*}
and a sufficient condition for $\dim_B(\cC) = \tau(\cC)$ to hold is that the
curve $\cC$ has what Aizenman and Burchard call the \textit{tempered crossing
property}, a property that restricts how wiggly a curve can be at small scales.

The main purpose of paper \cite{aizenman1999holder} is to develop tools to
study regularity and tightness of a family of random curves. For
$z \in \RR^2$ and $0 < l_1 < l_2$ define the annulus
$D(z; l_1, l_2) := \{w \in \RR^2; l_1 < |z - w| \le l_2\}$.
When $z$ is the origin, we simply write $D(l_1, l_2)$.
We refer to $\partial B(z, l_1)$ and $\partial B(z, l_2)$ as the
internal and external boundaries of $D(z; l_1, l_2)$.

\begin{defi}
\label{defi:path_traverses_annulus}
Given an annulus $D(z; l_1, l_2)$ and a random continuous curve $\gamma$,
let $f: [0,1] \to \RR^2$ be a parametrization of $\gamma$.
We say that $D(z; l_1, l_2)$ is
\textit{traversed by $k$ separate segments of $\gamma$} if
there are $k$ disjoint intervals $[t^{i}_1, t^{i}_2] \subset [0,1]$
for $i = 1, 2, \ldots, k$ such that for every $i$ it holds $f(t^{i}_1)$
belongs one of the boundaries $\partial B(z, l_j)$, $f(t^{i}_2)$
belongs to the other boundary of $D(z;l_1, l_2)$ and
$|f(t) - z| \in (l_1, l_2)$ for $t\in (t^{i}_1, t^{i}_2)$.
(see Figure~\ref{fig:traversing_annulus})
\end{defi}

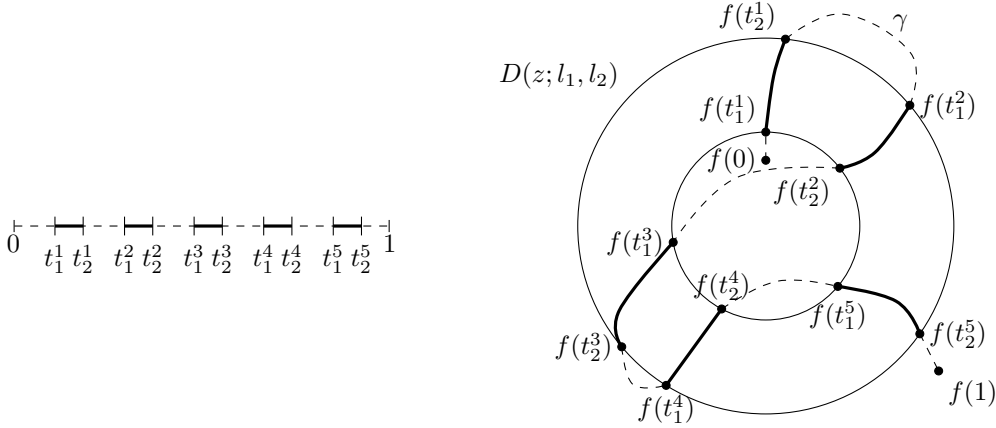
\begin{figure}
\centering
\begin{tikzpicture}[scale=.25]
%\clip (-15, -11) rectangle (15, 12);
\draw[|-|, dashed] (-40,0) node[below] {$0$} -- (-20,0) node[below] {$1$};
%\draw[->] (-15,0) -- (-12,0);
\draw (0,0) circle (5);
\draw (0,0) circle (10);
\draw[dashed] (90:3.5) -- (90:5);
\draw[very thick, tension=.7]
    plot [smooth] coordinates { (90:5) (87:8) (84:10) };% traversal 1
\draw[dashed, tension=.7]
    plot [smooth] coordinates {(84:10) (70:12) (50:12) (40:10)};
\draw[very thick, tension=.7]
    plot [smooth] coordinates {(40:10) (34:7) (38:5)}; % traversal 2
\draw[dashed, tension=.7]
    plot [smooth] coordinates {(38:5) (120:3) (190:5)};
\draw[very thick, tension=.7]
    plot [smooth] coordinates {(190:5) (210:9) (220:10)}; % traversal 3
\draw[dashed, tension=.7]
    plot [smooth] coordinates {(220:10) (230:11) (238:10)};
\draw[very thick, tension=.7]
    plot [smooth] coordinates {(238:10) (242:5)}; % traversal 4
\draw[dashed, tension=.7]
    plot [smooth] coordinates { (242:5) (282:3) (320:5) };
\draw[very thick, tension=.7]
    plot [smooth] coordinates {(320:5) (330:8) (325:10) };% traversal 5
\draw[dashed] (325:10) -- (320:12);

%labels
\node[right] at (60:12.5) {$\gamma$};
\node at (-11,  8) {$D(z; l_1, l_2)$};
\filldraw (90:3.5) circle (6pt) node[left]        {$f(0)$};
\filldraw (90: 5) circle (6pt) node[above left]   {$f(t_1^{1})$};
\filldraw (84:10) circle (6pt) node[above left]   {$f(t_2^{1})$};
\filldraw (40:10) circle (6pt) node[right]        {$f(t_1^{2})$};
\filldraw (38: 5) circle (6pt) node[below left]   {$f(t_2^{2})$};
\filldraw (190:5) circle (6pt) node[left]         {$f(t_1^{3})$};
\filldraw (220:10) circle (6pt) node[      left]  {$f(t_2^{3})$};
\filldraw (238:10) circle (6pt) node[below]       {$f(t_1^{4})$};
\filldraw (242: 5) circle (6pt) node[above]       {$f(t_2^{4})$};
\filldraw (320: 5) circle (6pt) node[below=2]     {$f(t_1^{5})$};
\filldraw (325:10) circle (6pt) node[      right] {$f(t_2^{5})$};
\filldraw (320:12) circle (6pt) node[below right] {$f(1)$};

\def\intervalLength{1.5}
\def\intervalSepAndLength{3.7}
\foreach \x in {1,2,3,4,5}{
    \draw (-40   + \x*\intervalSepAndLength,0.5) ++(-\intervalLength,0) -- ++(0,-1) node[below] {$t_1^{\x}$};
    \draw (-40   + \x*\intervalSepAndLength,0.5) -- ++(0,-1) node[below] {$t_2^{\x}$};
    \draw[very thick] (-40 + \x*\intervalSepAndLength,0) -- ++(-\intervalLength,0);
}

\end{tikzpicture}
\caption{Traversing annulus $D(z; l_1, l_2)$ with $5$
    separate segments of $\gamma$.}
\label{fig:traversing_annulus}
\end{figure}

In the following we consider a family $\{\gamma_r\}_{r \in (0,1]}$
of random curves that represent exploration paths. Parameter $r$ can be
interpreted as the truncation level above which we keep sticks of the soup.
Aizenman and Burchard define two hypotheses on a family of random curves.
For family $\{\gamma_r\}_{r \in (0,1]}$, the \textit{Hypothesis H1}
from~\cite{aizenman1999holder} is

\medskip
\noindent \textbf{H1.} For all $k < \infty$ and for all
annulus $D(z; l_1, l_2)$ with $l_1 \le l_2 \le 1$ we have that
\begin{equation}
\label{eq:hypothesis_H1_original}
    \sup_{r;\; 0 < r \le l_1 \le l_2 \le 1}
\PP\biggl(
    \begin{array}{l}
	\text{$D(z; l_1, l_2)$ is traversed by $k$ separate} \\
	\text{segments of curve $\gamma_r$}
    \end{array}
\biggr)
    \le K_k \left( \frac{l_1}{l_2}\right)^{\eta(k)}
\end{equation}
for some $K_k < \infty$ and some $\eta(k) \to \infty$ as $k \to \infty$.

\medskip
A family of random curves that satisfies \textbf{H1} has many useful
properties, cf.~\cite{aizenman1999holder}. The main technical step of
this section is to prove that a family $\{\gamma^0_r\}$ of exploration
paths satisfy this property, which is done in Theorem~\ref{teo:hypothesis_H1_IPS}.
As a consequence of Theorems~1.1 and 1.2 from~\cite{aizenman1999holder},
this implies that the family $\{\gamma^0_r\}$ is tight and gives an upper bound
on the Hausdorff dimension of limiting curves.

Aizenman and Burchard also define a \textit{Hypothesis H2}, which is sufficient
to ensure that the Hausdorff dimension of limiting curves is strictly larger
than 1. Basdevant, Blanc, Curien and Singh~\cite{basdevant} define a
related concept, \textit{Property ($\varnothing$)},
that is easier to verify for our model.
Let $\cF$ be a random closed subset of $[0,1]^2$. In words, Property
($\varnothing$) states that the probability of $\cF$ intersecting $n$
well-separated balls is exponentially small in $n$. Given $\zeta>1$ we say that
a collection of closed balls $(B(x_i, r_i); 1 \le i \le n)$ with $x_i \in
[0,1]^2$ and $r_i > 0$ is $\zeta$-separated if their dilations by $\zeta$ are
disjoint, that is: $\cap_{i=1}^n B(x_i, \zeta r_i) = \varnothing$.

\medskip
\noindent \textbf{Property ($\varnothing$).} The set $\cF$ satisfies Property
($\varnothing$) if there are constants $\zeta > 1$, $Q>0$ and $0 < q < 1$ such
that: for every $\zeta$-separated collection $(B(x_i, r_i); 1 \le i \le n)$ we
have
\begin{equation}
\label{eq:property_varnothing}
\PP(\text{$\cF$ intersects $B(x_i, r_i)$ for every $1 \le i \le n$})
    \le Q q^n.
\end{equation}

Theorem~3.2 of~\cite{basdevant} states that if $\cF$ satisfies Property
$(\varnothing)$ then there is a constant $c>1$ such that a.s. every connected
closed subset $\cC \subset \cF$ that is not a point or $\varnothing$ has
$\dim_\cH (\cC) > s > 1$. In Proposition~\ref{prop:property_varnothing} we show
exploration paths $\{\gamma^0_r\}$ contain subpaths that satisfy Property
($\varnothing$), implying the lower bound on Hausdorff dimension of limiting curves.

We end this section stating a more precise version of
Theorem~\ref{teo:tight_regularity_exploration_paths}.
%We say that a sequence of random variables $\{X_r\}_{0 < r \le 1}$
%is stochastically bounded as $r \to 0$ if for
%every $\varepsilon > 0$ there is $a < \infty$ with
%$\sup_{0<r\le 1}\PP[|X_r| \geq a] \le \varepsilon$.
%Also,
Recall that a family of
probability measures on a complete separable metric space is
relatively compact with respect to
convergence in distribution if and only if it is tight, 
cf.~\cite{billingsley2013convergence}. 

\begin{teo}
\label{teo:tight_regularity_exploration_paths_2}
Let $u \in (0, \bar{u})$ and consider a IPS $\xi$ of intensity $u$
and a box $B \subset \RR^2$. The family of exploration curves
$\{\gamma^0_r\}_{0 < r \le 1}$ in box $B$ a.s. has, for any
$\varepsilon > 0$, parametrizations $f_r: [0,1] \to \RR^2$
satisfying for $0 \le t_1 \le t_2 \le 1$ that
\begin{equation*}
|f_r(t_2) - f_r(t_1)|
    \le \kappa_{\varepsilon, r, B}(\xi) \cdot
        |t_1 - t_2|^{1/(2 - c(u) + \varepsilon)}
\end{equation*}
with $\kappa_{\varepsilon, r, B}$ a tight family of random variables
as $r \to 0$ and $c(u) > 0$, implying that the family
$\{\gamma_r^0\}$ is tight. Moreover, any subsequential limit
$\cC$ is supported on curves of $\cS_B$ with
$\alpha(\cC)^{-1} = \dim_B \cC \le 2 - c(u)$ and
$1 < \dim_\cH \cC \le 2 - 2 c(u)$.
\end{teo}

\subsection{Defining box-crossing events}
\label{sub:defining_box_crossing}

Let us discuss how one can define the event `there is a vacant crossing of a
box' in an IPS. For $z,w \in \RR^{2}$ let $[z, w]$ denote the line segment
with endpoints $z$ and $w$, and let $(z,w) := [z,w] \setminus \{z, w\}$ be its
\textit{interior}.

It is convenient to allow curves that may intersect $\cE_0$. Indeed, 
take any decreasing sequence $\varepsilon_n \downarrow 0$. Notice that if
$\cC_n$ is a sequence of curves contained in $\cV_{\varepsilon_n}$ that
converges to some curve $\cC$, then it is possible that
$\cC \cap \cE_0 \neq \varnothing$. For that reason, our defintion of a
curve crossing a segment must be more refined than simply intersecting the
segment.

Given $[z, w] \subset \RR^2$, denote by $\ell$ the infinite line supporting
$[z,w]$. Line $\ell$ divides the plane into two open half-planes, let us call
them $H_1$ and $H_2$.  Intuitively, if a curve $\cC$ passes through a point $o
\in (z,w)$ we can say that it crosses $[z,w]$ if it arrives by one of the
half-planes and leaves at the other one, see Figure
\ref{fig:non_crossing_limit_curve}.

Let $f$ be some parametrization of a curve $\cC$.
Consider $t_0 \in [0,1]$ with $f(t_0) \in \ell$.
The function $d : t \mapsto \dist(f(t), \ell)$ is
continuous and $t_0 \in d^{-1}(0)$. Let us define
\begin{equation*}
    t_0^- := \,\inf\{t \le t_0;\; [t,t_0] \in d^{-1}(0)\}
    \quad \text{and} \quad
    t_0^+ := \sup\{t \ge t_0;\; [t_0,t] \in d^{-1}(0)\},
\end{equation*}
so that $[t_0^-,t_0^+]$ is the largest interval such that
$f([t_0^-,t_0^+]) \subset \ell$.
We say that $t_0$ is a \textit{left-access point} if $t_0 = t_0^- \neq 0$ and
is a \textit{right-access point} if $t_0 = t_0^+ \neq 1$. Notice that if $t_0$
is a left-access point then for all $\delta>0$ we have
    \begin{equation}
    \label{eq:left-access}
    \text{$f\bigl((t_0-\delta, t_0)\bigr) \cap H_i \neq \varnothing$ for some
        $H_i$}.
    \end{equation}
        
We would like to identify from which half-space the function $f$ reached a
left-access point. However, since continuous curves can be rather erratic,
there might be no unique answer. We say that $t_0$ is a
\textit{left-oscillation point} if for all $\delta>0$, we have
    \begin{equation}
    \label{eq:left-oscillation}
    \text{$f\bigl((t_0-\delta, t_0)\bigr) \cap H_1 \neq \varnothing$ and 
    $f\bigl((t_0-\delta, t_0)\bigr) \cap H_2 \neq \varnothing$.}
    \end{equation}
Any left-access point that is not a left-oscillation point can be associated to
a unique half-plane, since for $\delta>0$ small there is $i_0$ such
that $f\bigl((t_0^--\delta, t_0^-)\bigr) \cap H_{i_0} \neq \varnothing$ but 
$f\bigl((t_0^--\delta, t_0^-)\bigr) \cap H_{3- i_0} = \varnothing$; let us define
$s^{-}(t_0^-) := H_{i_0}$. The discussion of right-access and right-oscillation
points is entirely analogous, replacing intervals $(t_0^--\delta, t_0^-)$ by intervals
$(t_0^+, t_0^++\delta)$ and denoting the half-plane associated to a right-access point
$t_0^+$ (that is not a right-oscilation point) by $s^+(t_0^+)$.

\begin{defi}[Crossing of a segment by a curve]
\label{defi:cross_segment_by_curve}
We say a curve $\cC$ \textit{crosses} $[z,w]$ if for some parametrization
$f$ either of the two scenarios happen:
\begin{itemize}
\item[\textbf{(S1)}]
    There is either a left-oscillation or a right-oscillation point $t_0 \in [0,1]$ with $f(t_0) \in (z,w)$.
\item[\textbf{(S2)}]
    There is $t_0 \in [0,1]$ with $f\bigl([t_0^-,t_0^+]\bigr) \subset (z,w)$ and 
        $s^-(t_0^-) \neq s^+(t_0^+)$.
\end{itemize}
Notice that $\cC$ crosses $[z,w]$ is independent of the parametrization.
\end{defi}

\begin{defi}[Vacant crossing for IPS]
Fix a rectangular box $B \subset \RR^2$ parallel to the coordinate axes.
For $r\geq 0$, define $\overline{LR}_r(B)$ as the
event that there is a curve $\cC$ contained in $B$ with one endpoint
on the left side and the other on the right side of $B$ that does not
cross any stick from $\xi_r$.
\end{defi}

\begin{remark}
\label{rem:measurability_box_crossing}
For any $r > 0$ the event $\overline{LR}_r$ is measurable,
since a.s. any box intersects a finite number of sticks from $\xi_r$.
Measurability of $\overline{LR}_0$ can be seen as a consequence of our results,
see Proposition~\ref{prop:LR0_measurable}.
\end{remark}

A useful consequence of Definition~\ref{defi:cross_segment_by_curve} is

\begin{prop}
\label{prop:curve_crosses_stick}
Let $\cC_n$ be a sequence of curves that do not cross $[z,w]$ and suppose that
$\cC_n$ converges to $\cC$ in the metric $d$. Then, $\cC$ does not cross $[z,w]$.
\end{prop}

\begin{proof}
Since $\cC_n$ converges to $\cC$, we have a sequence of parametrizations
$f_n$ of the curves $\cC_n$ such that $f_n \to f$ uniformly,
where $f$ is a parametrization of $\cC$. Suppose by contradiction that
$\cC$ crosses $[z, w]$. We check either scenario separately.

\medskip
\noindent
\textbf{Scenario (S1).} There is some oscillation point, say $t_0$
is a right-oscillation point with $f(t_0) \in (z,w)$.
Define $\varepsilon := \tfrac{1}{4} \dist(f(t_0), \{z,w\})$ and take
$\delta > 0$ sufficiently small so that
$f\bigl([t_0, t_0+\delta]\bigr) \subset B(f(t_0), \varepsilon)$.
By uniform convergence, we have for $t \in [t_0, t_0+\delta]$ and $n$
large enough that
\begin{equation*}
\dist(f_n(t_0), f_n(t))
    \le \dist(f_n(t_0), f(t_0)) + \dist(f(t_0), f(t)) + \dist(f(t), f_n(t))
    \le 3 \varepsilon,
\end{equation*}
which implies $f_n([t_0, t_0+\delta]) \subset B(f(t_0), 3\varepsilon)$. Since
$t_0$ is a right-oscillation point, we can find points
$t_i \in [t_0,t_0+\delta]$ with $f(t_i) \in H_i$ for $i=1,2$ with $t_1 < t_2$
and taking $n$ larger if needed, we have $f_n(t_i) \in H_i$. Since
$f_n([t_1,t_2])$ is a connected subset of $B(f(t_0), 3\varepsilon)$ that
intersects both $H_1$ and $H_2$, it must intersect $(z,w)$. Finally, notice
that $f_n$ must contain either an oscillation point or a point 
$\tilde{t} \in [t_1, t_2]$ with $f\bigl([\tilde{t}^-,\tilde{t}^+]\bigr)$ and
$s_n^-(\tilde{t}^-) \neq s_n^+(\tilde{t}^+)$, otherwise every time $f_n$
intersected $(z,w)$ it would enter and leave by the same half-plane. We have
reached a contradiction.

\medskip
\noindent
\textbf{Scenario (S2).} There is $t_0 \in [0,1]$ with
$f\bigl([t_0^-,t_0^+]\bigr) \subset (z,w)$ and $s^-(t_0^-) = H_1$ and
$s^+(t_0^+) = H_2$, without loss of generality.
Define
$\varepsilon = \tfrac{1}{4} \, \dist(f([t_0^-, t_0^+]), \{z, w\})$.
Taking $\delta>0$ small, we have that
\begin{equation*}
    f([t_0^- - \delta, t_0^+ + \delta])
    \subset f([t_0^-, t_0^+])^{+\varepsilon}
    := \{y \in \RR^2;\ \dist(y, f([t_0^-, t_0^+])) \le \varepsilon\}.
\end{equation*}
As before, using uniform convergence we can ensure
$f_n\bigl([t_0^- -\delta, t_0^+ +\delta]\bigr) \subset f([t_0^-, t_0^+])^{+2\varepsilon}$
for large $n$.
%\begin{equation*}
%\dist(f([t_1, t_2]), f_n(t))
    %\le \dist(f([t_1, t_2]), f(t)) +
        %\dist(f(t), f_n(t))
    %\le 2\varepsilon,
%\end{equation*}
By definition of $t_0$, we can find $t_1 \in (t_0^- - \delta)$ and $t_2 \in
(t_0^+ + \delta)$ with $f(t_i) \in H_i$ and taking $n$ sufficiently large, we
can ensure that $f_n(t_i) \in H_i$ for $i=1,2$. Finally, notice that $[z, w]$
divides $f([t_1, t_2])^{+2\varepsilon}$ into two connected components, like in
scenario \textbf{(S1)}. Analogously, since $\cC_n$ does not cross $[z,w]$ by
hypothesis, we have a contradiction.
\end{proof}

\begin{figure}
\centering
\begin{tikzpicture}[scale=.85]
\draw[|-|] (-13,0) node[below] {$w$} -- (-5, 0) node[below] {$z$};

\filldraw (-10,0) circle (2pt);
\filldraw (-9.6,-.2) circle (2pt);
\filldraw (-9.4,.4) circle (2pt);
\node at (-10.6,-.4) {$f(t_0)$};
\draw[dashed] (-10,0) circle (1.3);

\draw[line width=2] % t_0 oscillation point
    (-10,0) -- ++(.2,.2) --  ++(.2,-.4) node[below] {$f(t_1)$}
    --  ++(.2,.6) node[above left] {$f(t_2)$};

\draw[|-|] (-4,0) node[below] {$w$} -- (4, 0) node[below] {$z$};
\draw[dashed] (2,1) -- (-2,1) 
    node[above] {$f([t_0^-, t_0^+])^{+2\varepsilon}$}
    arc (90:270:1) -- (2,-1) arc (-90:90:1);
\draw[line width=2] (-2,0) -- (2,0);

\draw[line width=2, tension=.7] % t_2^+
    plot [smooth] coordinates
    {(1.2,0) (1,.4) (1,.7)} node[right] {$f(t_2)$};
    \filldraw (1,.7) circle (2pt);
\draw[line width=2, tension=.7] % t_1^-
    plot [smooth] coordinates
    {(-1.3,0) (-1.5, -.2)(-1,-.4) (-1,-.7)} node[right] {$f(t_1)$};
    \filldraw (-1,-.7) circle (2pt);
\end{tikzpicture}
\caption{Constructions on Proposition \ref{prop:curve_crosses_stick}}
\label{fig:non_crossing_limit_curve}
\end{figure}
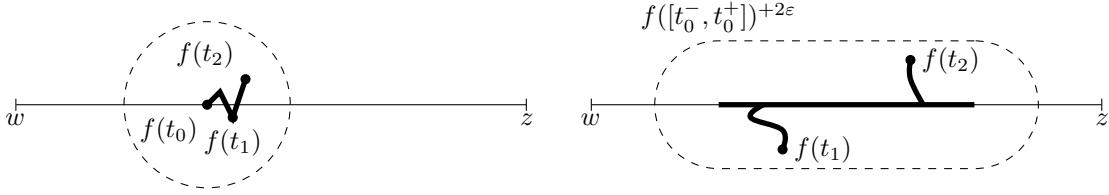

\begin{coro}
\label{coro:curve_crosses_stick_soup}
Let $\cC_n$ be a sequence of curves that do not cross any stick of a collection
$\cE^{(n)} = \{[z_i, w_i]\}$. Also, suppose that $\cC_n$ converge to $\cC$
in the metric $d$ and that $\cE^{(n)}$ is increasing. Then, $\cC$ does not
cross sticks from $\cE = \cup \cE^{(n)}$.
\end{coro}

%\begin{proof}
%Straightforward from Proposition~\ref{prop:curve_crosses_stick}.
%\end{proof}

\begin{remark}
\label{rem:topological_issue}
By definition, on event $\overline{LR}_r(B)$ there is certainly some curve
$\cC_r$ that makes the crossing of $B$ without crossing sticks in $\xi_r$.
However, it is possible that we are on event $\cap_{r > 0} \overline{LR}_r(B)$
but not on event $\overline{LR}_0(B)$. There are configurations with
curves $\cC_r$ that do not converge to a curve when $r \to 0$. This phenomenon
is related to the fact that the intersection of non-empty compact
path-connected sets in the plane is a continuum, i.e., compact and connected,
but is not necessarily path-connected. For a concrete example, consider the
topologist's sine curve
\begin{equation*}
K := \{(x, \sin (1/x)); \ 0 < x \le 1\} \cup (\{0\}\times[-1,1])
\end{equation*}
and notice $K = \cap_{r>0} K^{+r}$, where
$K^{+r} := \{y \in \RR^2;\ \dist(y, K) \le r\}$.
For box $B = [0,1]\times[-1,1]$ we can easily build a path
connecting its left and right sides entirely contained in $B\cap K^{+r}$
for a fixed $r>0$. However, there is no path connecting
$\{0\}\times [-1,1]$ to $(1, \sin 1)$ contained in $K$.
This represents no real issue, since almost surely events
$\cap_{r > 0} \overline{LR}_r(B)$ and $\overline{LR}_0(B)$ are equal, see
Proposition~\ref{prop:LR0_measurable}.
\end{remark}

\subsection{Exploration paths and their properties}
\label{sub:exploration_paths_and_their_properties}

We discuss exploration processes for both the ellipses model and the IPS.
Such processes are easier to define on ellipses model and explorations on
IPS can be seen as limits of explorations on ellipses models with minor axis
tending to 0.

Recall that $\xi_{r}$ is the restriction of $\xi$ to sticks with
radius at least $r$.
In this way, we have a coupling for all $\xi_r$ with $r \geq 0$. Also, recall
that for fixed $r>0$ we can obtain either a collection of ellipses or
sticks by considering $\cE_r(\xi_r)$ and $\cE_0(\xi_r)$. We do
not define an exploration path on the full soup; instead, it will be seen
as a limit from the excursions on $\cE_0(\xi_r)$.

\medskip
\noindent
\textbf{Ellipses model.}
On this framework, let us define exploration paths for $\cE_r^l = \cE_l(\xi_r)$
where $l \in (0,r]$. Fix a box $B$ (with sides parallel to $xy$ axes,
without loss of generality) and impose that the bottom side of $B$ is
covered, while its left side is vacant. The exploration path
$\gamma_r^l$ is a random path defined for each configuration of the
ellipses model. It begins at the lower left corner of
$B$ and follows the interface between the covered set $\cE_r^l$
and its complement, keeping the covered region always at its right.
The path is allowed to walk over the left and bottom sides of the
box and also over the arcs of ellipses intersecting the box until
it meets either the top side or the right side of $B$
(see Figure~\ref{fig:exploration_ellipses}). Since we know that for
any $l \in (0, r]$ almost surely there is only a finite number of
ellipses in $\cE_r^l$ intersecting $B$, the exploration paths
$\gamma_r^l$ are well-defined. Also, we can parametrize curve
$\gamma_r^l$ to ensure it has no self-intersections.

%%%%%%%%%%%%%%%%%%%%%%%%%%%%%%%%%%
%drawing on the left: ellipses exploration
%drawing on the right: sticks exploration
%%%%%%%%%%%%%%%%%%%%%%%%%%%%%%%%%%
\begin{figure}[ht]
\centering
\begin{tikzpicture}[baseline=(current bounding box.north),scale = 1.5,
   occupied/.style={pattern=north east lines,pattern color=gray!80}
]

\draw           (-2,-1  ) rectangle (2, 1); %box
\fill[occupied] (-2,-1.3) rectangle (2,-1);
\draw[dashed]   (-2.3,-1) rectangle (-2,1);
\path[name path = bottom] (-2,-1) -- (2,-1); %auxiliary lines from box
\path[name path = left]   (-2,-1) -- (-2,1); %auxiliary lines from box
\path[name path = right]  ( 2,-1) -- ( 2,1); %auxiliary lines from box

%text
 \node at (0,-1.15) {Covered};
 \node[rotate = 90] at (-2.15,0) {Vacant};

%defining the ellipses intersecting box that form the exploration path
 \path[name path = s1, rotate around={ 50:(-1.1, -1.1)}] (-1.1,-1.1) circle (1 and .1);
 \path[name path = s2, rotate around={-10:(-1.1, -0.5)}] (-1.1,-0.5) circle (1 and .1);
 \path[name path = s3, rotate around={ 15:(-1.7, -0.5)}] (-1.7,-0.5) circle (.7 and .1);
 \path[name path = s4, rotate around={ 90:(-1.1, -0.1)}] (-1.1,-0.1) circle (.3 and .1);
 \path[name path = s5, rotate around={  5:(   0,  0.2)}] (   0, 0.2) circle (1.2 and .1);
 \path[name path = s6, rotate around={100:( 1.2,  0.2)}] ( 1.2, 0.2) circle (.6 and .1);
 \path[name path = s7, rotate around={  3:( 1.6, -0.2)}] ( 1.6,-0.2) circle (.6 and .1);

%drawing the ellipses hiding what is outside the box
\begin{scope}
\clip (-2,-1  ) rectangle (2, 1);

 \draw[rotate around={ 50:(-1.1, -1.1)}] (-1.1,-1.1) circle (1 and 0);
 \draw[dashed, rotate around={ 50:(-1.1, -1.1)}] (-1.1,-1.1) circle (1 and .1);
 \draw[rotate around={-10:(-1.1, -0.5)}] (-1.1,-0.5) circle (1 and 0);
 \draw[dashed, rotate around={-10:(-1.1, -0.5)}] (-1.1,-0.5) circle (1 and .1);
 \draw[rotate around={ 15:(-1.7, -0.5)}] (-1.7,-0.5) circle (.7 and 0);
 \draw[dashed, rotate around={ 15:(-1.7, -0.5)}] (-1.7,-0.5) circle (.7 and .1);
 \draw[rotate around={ 90:(-1.1, -0.1)}] (-1.1,-0.1) circle (.3 and 0);
 \draw[dashed, rotate around={ 90:(-1.1, -0.1)}] (-1.1,-0.1) circle (.3 and .1);
 \draw[rotate around={  5:(   0,  0.2)}] (   0, 0.2) circle (1.2 and 0);
 \draw[dashed, rotate around={  5:(   0,  0.2)}] (   0, 0.2) circle (1.2 and .1);
 \draw[rotate around={100:( 1.2,  0.2)}] ( 1.2, 0.2) circle (.6 and 0);
 \draw[dashed, rotate around={100:( 1.2,  0.2)}] ( 1.2, 0.2) circle (.6 and .1);
 \draw[rotate around={  3:( 1.6, -0.2)}] ( 1.6,-0.2) circle (.6 and 0);
 \draw[dashed, rotate around={  3:( 1.6, -0.2)}] ( 1.6,-0.2) circle (.6 and .1);

%decorative ellipses
\draw[rotate around={110:(1,-1)}] ( 1,-1) circle (1 and 0);
\draw[dashed, rotate around={110:(1,-1)}] ( 1,-1) circle (1 and .1);
\draw[rotate around={60:(.7,-1)}] (.7,-1) circle (.6 and 0);
\draw[dashed, rotate around={60:(.7,-1)}] (.7,-1) circle (.6 and .1);
\draw[rotate around={-6:(-1.5,.6)}] (-1.5,.6) circle (.4 and 0);
\draw[dashed, rotate around={-6:(-1.5,.6)}] (-1.5,.6) circle (.4 and .1);
\end{scope}

%identifying intersections TO SHOW LABEL, ADD [label=above:A]
\path [name intersections={of = s1 and bottom}];
	\coordinate  (A)  at (intersection-1);
\path [name intersections={of = s1 and s2}];
	\coordinate  (B)  at (intersection-2);
\path [name intersections={of = s2 and s3}];
	\coordinate  (F)  at (intersection-1);
	\coordinate  (DB) at (intersection-3);
    \coordinate  (C)  at (intersection-4);
\path [name intersections={of = s3 and left}];
	\coordinate  (DA) at (intersection-1);
	\coordinate  (D)  at (intersection-2);
\path [name intersections={of = s2 and left}];
	\coordinate  (E)  at (intersection-1);
	\coordinate  (DC) at (intersection-2);
\path [name intersections={of = s3 and s4}];
	\coordinate  (G)  at (intersection-1);
\path [name intersections={of = s4 and s5}];
	\coordinate  (H)  at (intersection-3);
\path [name intersections={of = s5 and s6}];
	\coordinate  (I)  at (intersection-1);
\path [name intersections={of = s6 and s7}];
	\coordinate  (J)  at (intersection-3);
\path [name intersections={of = s7 and right}];
	\coordinate  (K)  at (intersection-1);

%exploration path
\draw[very thick] (-2,-1) --
    (A) -- (B) -- (C) -- (D) --
    (DA) -- (DB) -- (DC) -- (E)
    plot [tension=.7, smooth] coordinates {(E) (-1.8, -.3) (F)}
    -- (G) arc (220:60:.1 and .3)
    plot [tension=.7, smooth] coordinates
    {(H) (0, .3) ($(I)+(-.15, .01)$) (I)};
\begin{scope} %really ugly solution to make ellipse
\newcommand*{\initialangle}{75}
\newcommand*{\finalangle}{-125}

%%%%%%%%%%%%%%%%%%%
%Gives length of vector \p1 THIS IS WHERE THE 5.20348pt COMES FROM
%
%\draw let
%  \p1 = (\initialangle:.6 and .1),
%  \n1 = {veclen(\x1,\y1)}
%in
%   node{\n1};

\begin{scope}[shift={($(I) - (0,0)!5.20348pt!100:(\initialangle:.6 and .1)$)}]
\draw[very thick, rotate around={100:(0,0)}]
    (\initialangle:.6 and .1) arc (\initialangle:\finalangle:.6 and .1);
\end{scope}
\end{scope}
\draw[very thick] plot [tension=.7, smooth] coordinates {(J) (1.6, -.1) (K)};
\end{tikzpicture}
\hfill%
\begin{tikzpicture}[baseline=(current bounding box.north),scale = 1.5,
   occupied/.style={pattern=north east lines,pattern color=gray!80}
]

\draw           (-2,-1  ) rectangle (2, 1);
\fill[occupied] (-2,-1.3) rectangle (2,-1);
\draw[dashed]   (-2.3,-1) rectangle (-2,1);

%text
 \node at (0,-1.15) {Covered};
 \node[rotate = 90] at (-2.15,0) {Vacant};

%sticks intersecting box that form the exploration path
\clip (-2,-1  ) rectangle (2, 1);

 \draw[name path = s1, rotate around={ 50:(-1.1, -1.1)}] (-1.1,-1.1) circle (1 and 0);
 \draw[name path = s2, rotate around={-10:(-1.1, -0.5)}] (-1.1,-0.5) circle (1 and 0);
 \draw[name path = s3, rotate around={ 15:(-1.7, -0.5)}] (-1.7,-0.5) circle (.7 and 0);
 \draw[name path = s4, rotate around={ 90:(-1.1, -0.1)}] (-1.1,-0.1) circle (.3 and 0);
 \draw[name path = s5, rotate around={  5:(   0,  0.2)}] (   0, 0.2) circle (1.2 and 0);
 \draw[name path = s6, rotate around={100:( 1.2,  0.2)}] ( 1.2, 0.2) circle (.6 and 0);
 \draw[name path = s7, rotate around={  3:( 1.6, -0.2)}] ( 1.6,-0.2) circle (.6 and 0);

%exploration path
\draw[very thick] (-2,-1) -- (-1,-1) -- ++(50:.53) --
    ++(170:.83) -- (-2,-.6); %part 1
\draw[very thick] (-2,-.33) -- ++(-10:.54) -- ++(15:.4) --
    ++(90:.42) -- ++(185:.12) -- ++(5:.12) -- ++(90:.1) --
    ++(-90:.1) -- ++(5:2.28) -- ++(100:.51) --
    ++(-80:1.03) -- (2,-.18); %part 2

%decorative sticks
\draw[rotate around={110:(1,-1)}] ( 1,-1) circle (1 and 0);
\draw[rotate around={60:(.7,-1)}] (.7,-1) circle (.6 and 0);
\draw[rotate around={-6:(-1.5,.6)}] (-1.5,.6) circle (.4 and 0);
\end{tikzpicture}
\caption{Exploration paths for an ellipses model and its stick soup.
The latter exploration path can be seen as an exploration path for
ellipses of minor axis size equal to zero and can have self
intersections.}
\label{fig:exploration_ellipses}
\end{figure}

Following the exploration path $\gamma_r^l$, we have only two possible
outcomes. If $\gamma_r^l$ ends at the right side of $B$, then there is a
left-right vacant crossing of $B$. On the other hand, if $\gamma_r^l$ ends at
the top side there is a top-bottom covered crossing of the box. Intuitively, we
are looking for the lowest crossing of the box (or, if it does not exist, the
leftmost dual crossing). Notice also that curves
$\gamma_r^l$ do not touch the interior of any ellipse of the soup, with
the only exceptions being when it walks over the left side of $B$. Denote by
$\tilde{\gamma}_r^l$ the subpath of $\gamma_r^l$ that starts at the
last time $\gamma_r^l$ touches the left side of $B$. If path $\gamma_r^l$
does cross $B$ then subpath $\tilde{\gamma}_r^l$ is the lowest crossing of $B$.

\medskip
\noindent
\textbf{IPS.}
We can see the exploration path $\gamma_r^0$ as a degenerate case of the
ellipses model exploration path $\gamma_r^l$, in which $l=0$. Notice that
exploration paths $\gamma_r^l$ converge to a curve when $l\downarrow 0$.
This claim is quite easy to
believe and we do not provide a formal proof, just discuss
why the result is true. If you fix a finite configuration of
sticks $(s_i)$ in $B$ and take
\begin{equation*}
l < \min\{\dist(E_0(s_i), E_0(s_j)); \ E_0(s_i)\cap E_0(s_j) =\varnothing \ \forall i \neq j\},
\end{equation*}
then the covered connected components of $\cup_i E_0(s_i)$ and $\cup_i
E_l(s_i)$ are close to one another
(check Figure~\ref{fig:exploration_ellipses}).

Notice that curve $\gamma_r^0$ obtained in the limit can touch the left sides
of $B$ and can walk over its bottom side and also over the sticks intersecting
the box. Such curves can have self-intersections, but any point is visited at
most four times. Similarly to paths $\gamma_r^l$ with $l>0$, curve $\gamma_r^0$ cannot
cross any stick of the soup in general (except sticks touching the left side of $B$)
and defining $\tilde{\gamma}_{r}^{0}$ as the subcurve of $\gamma_r^0$
starting from the last visit to the left side of $B$, we obtain that
\begin{equation}
\label{eq:path_from_last_left}
\text{$\tilde{\gamma}_{r}^{0}$ is a path that does not cross any stick of
    the soup}.
\end{equation}

If $\tilde{\gamma}_r^0$ ends on the top side, we can conclude that there was no
left-right vacant crossing of the box, even in the non-truncated IPS. On the
other hand, if $\tilde{\gamma}_r^0$ ends on the right side, we can only
conclude there was a crossing when ignoring sticks with
radius smaller than $r$, i.e., that event $\overline{LR}_r(B)$ happened.
Our goal now is to ensure that the family of curves $\{\gamma_r^0\}_{0 < r
\le 1}$ satisfies \textbf{H1}, which is done in
Section~\ref{sub:proof_of_hypothesis_h1}. As a consequence, 
Theorem~\ref{teo:tight_regularity_exploration_paths} holds, implying the
family is tight and its subsequential limits are regular.
Moreover, it follows that
\begin{prop}
\label{prop:LR0_measurable}
We have that $\overline{LR}_0(B) = \cap_{r>0} \overline{LR}_r(B)$ almost
    surely.
\end{prop}
\begin{proof}
The family $\{\tilde{\gamma}_r^0\}_{0 < r \le 1}$ of paths starting from the last
visit to the left side of $B$ also satisfies
Theorem~\ref{teo:tight_regularity_exploration_paths}. By
Corollary~\ref{coro:curve_crosses_stick_soup} and
\eqref{eq:path_from_last_left}, any subsequential limit is a random curve
$\gamma$ that does not cross any sticks of the soup and crosses $B$ from
left to right. This justifies our claim that $\overline{LR}_0(B)$ is the
same event as $\cap_{r>0} \overline{LR}_r(B)$ apart from a zero measure set.
\end{proof}

We end this section collecting some useful facts about exploration paths
$\gamma_r^l$. In order to prove some properties of exploration paths, we
introduce notation for paths and concatenation. If $p, q \in \gamma$, we denote
by $p\gamma q$ the curve from $p$ to $q$ following $\gamma$. For points $p, q$
in the boundary of a convex $C \subset \RR^2$ with non-empty interior, we
denote by $pCq$ the clockwise arc of $\partial C$ from $p$ to $q$ and by
$pC^-q$ the anti-clockwise arc from $p$ to $q$. Concatenation of paths is
represented by combining these notations, eg. $a\gamma bEc$.

On the course of proving \textbf{H1}, we will be interested in checking if
exploration path $\gamma_r^l$ traverses a fixed annulus $D = D(z; l_1, l_2)$
many times or not, as in Definition~\ref{defi:path_traverses_annulus}. Define
$\partial_i D := z + \partial B(l_i)$ for $i=1,2$.

\begin{defi}[Entering arms]
Fix an annulus $D$.
Let $f:[0,1] \to \RR^2$ be an injective parametrization of $\gamma =
\gamma_r^l$ with $f(0)$ being the lower left corner of $B$. We say a
segment of the curve $f([t_1, t_2])$ with $0 \le t_1 < t_2 \le 1$ is an
\textit{entering arm} for $D$ if $f(t_1) \in \partial_2 D$, $f((t_1, t_2))$
is in the interior of $D$ and $f(t_2) \in \partial_1 D$. Interchanging the
roles of $\partial_1 D$ and $\partial_2 D$ we have the definition of an
\textit{exiting arm} for $D$. An ellipse $E$ is used by an entering (or
exiting) arm $f([t_1, t_2])$ if their intersection is non-empty.
\end{defi}

The next lemma shows that an exploration path that uses an ellipse must obey
some restrictions of topological character. The most important property in what
follows is described by \textbf{P3} below. It ensures that within a given
annulus each stick can participate in at most two entering arms, with equality
only if the stick intersects the inner boundary twice
(see also Figure~\ref{fig:P3_construction}).

\begin{lema}[Topological properties of exploration paths]
\label{lema:topological_properties}
Fix $0 < l \le r$. Let $\gamma = \gamma_r^l$ be the
exploration path of the box $B$ and $E$ be any ellipse of the process
intersecting $B$. The following properties hold
\begin{itemize}
\item[\textup{\textbf{P1.}}] Ellipse $E$ is touched in clockwise
order by the exploration path $\gamma$.

\item[\textup{\textbf{P2.}}] Fix an annulus $D = D(z; l_1, l_2)$  with
$D \cap E \neq \varnothing$. Let $\cP_1$ and $\cP_2$ be two entering
arms for $D$ that use $E$ and let $v_i$ be the first
visit of $\cP_i$ to $E$. If $\cP_1$ and $\cP_{2}$ exist then
$\partial_1 D$ must intersect $v_1 E v_{2}$ and
$v_{2} E v_{1}$.

\item[\textup{\textbf{P3.}}] If $E$ is used by $\gamma$ to 
enter $D(z; l_1, l_2)$ then $E$ can only be used by an entering
arm of $\gamma$ again if $\partial_1 D\backslash E$ has
two connected components. Moreover, $E$ cannot be used by
three entering exploration arms. Finally, this property also holds
for $l=0$.
\end{itemize}
\end{lema}

\begin{proof}
To prove \textbf{P1}, fix an injective parametrization $f:[0,1] \to \RR^2$ of
$\gamma$. Suppose that for $0 \le t_1 < t_2 \le 1$ we have
$f(t_1), f(t_2) \in \partial E$ and $f((t_1, t_2)) \subset \RR^2\backslash E$.
If we consider the Jordan curve $\cC= f(t_1)\gamma f(t_2) E^- f(t_1)$, then
$f((t_2, 1])$ cannot intersect $\cC$ or its interior. This implies that if
$f$ returns to $\partial E$, it must be on a point in $f(t_2) E f(t_1)$ and the
proof of \textbf{P1} is over.

The proof of \textbf{P2} is also built upon Jordan curves. Notice that
almost surely, we have that $E$ is not tangent to $\partial_1 D$. 
Let $p^{i}_{j}$ be the intersection point of $\cP_i$ with
$\partial_{j} D$, for $j=1,2$.
Suppose $\cQ_1 = v_1 E v_{2}$ does not intersect $\partial_1 D$. There are two
cases to consider and in each one we find a Jordan curve $\cC^{\ast}$ such that
\begin{enumerate}[(i)]
\item We have $\smash{p^{1}_{1}}$ inside $\smash{\cC^{\ast}}$ and
    $\smash{p^{2}_{1}}$ outside $\smash{\cC^{\ast}}$;
\item $\smash{\cC^{\ast}}$ does not intersect $\smash{\partial_1 D}$.
\end{enumerate}
Noticing that $p^{j}_1$ are points of $\partial_1 D$ we reach a contradiction,
    since $\partial_1 D$ is path-connected implies it must cross $\cC^{\ast}$
    (see Figure~\ref{fig:P2_construction}).

\medskip
\noindent
\textbf{Case $\cQ_1$ intersects $\partial_2 D$.}
Let $q$ be the first point of $\partial_2 D$ that intersects $\cQ_1$. Then,
$\cC^{\ast} = p^{1}_{2} \gamma v_1 E q \partial_2D^{-} p^{1}_2$ is a Jordan
curve that does not intersect $\partial_1D$.

\medskip
\noindent
\textbf{Case $\cQ_1$ does not intersect $\partial_2 D$.}
Take
$\cC^\ast
    = p^{1}_2 \gamma v_1 E v_{2} \gamma^{-}
      p^{2}_2 \partial_2 D^{-} p^{1}_2$.

\medskip
A similar reasoning implies that $\partial_1 D$ intersects
$v_{2} E v_1$ and we can conclude that \textbf{P2} holds.

Finally, we prove \textbf{P3}. Notice that $v_i$ cannot be inside
$\partial_1D$ since then $p^{i}_2 \gamma v_i$ would intersect
$\partial_1D$. Then, the lemma follows from \textbf{P2} and the
fact that $B(l_1)\backslash E$ can have at most 2 connected components.
The case $l=0$ can be deduced by making $l \downarrow 0$, since
$\gamma_r^l \to \gamma_r^0$.
\end{proof}

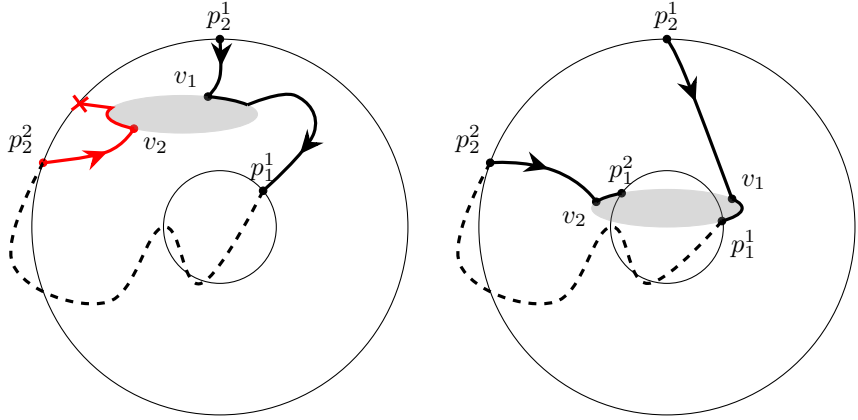
\begin{figure}
\centering
\begin{tikzpicture}[scale=0.25,
    dot/.style={
        draw,circle,minimum size=1mm,inner sep=0pt,outer
        sep=0pt,fill=black},
    filling/.style={gray, opacity=.3}
    ]
    \coordinate          (c) at (-2, 6); % center of ellipse
    \coordinate [dot]  (p12) at (  90:10)
                        node at  (p12) [above] {$p^{1}_2$};
    \coordinate [dot, red]  (p22) at ( 160:10)
                        node at (p22) [above left] {$p^{2}_2$};
    \coordinate [dot]   (v1) at ($(c)+(70:4 and 1)$)
                        node at  (v1) [above left] {$v_1$};
    \coordinate        (v1l) at ($(c)+(30:4 and 1)$); %left E after (v1)
    \coordinate [dot, red]   (v2) at ($(c)+(-130:4 and 1)$)
                        node at (v2) [below right] {$v_{2}$};
    \coordinate        (v2l) at ($(c)+(-200:4 and 1)$); %left E after (v2)
    \coordinate [dot]  (p11) at (40:3)
                        node at  (p11) [above] {$p^{1}_1$};
    %\coordinate [dot]  (p21) at (160:3)
                        %node at (p21) [below] {$p^{2}_1$};
    % circles and elipse
    \draw (0,0) circle (10);
    \draw (0,0) circle (3);
    \filldraw[filling] (c) circle (4 and 1);
    %path cP_1
    \draw[very thick]
        plot [smooth, tension=.7]
        coordinates {(p12) (0,8) (v1)}
        [arrow inside={end={Stealth[width=8pt, length=9pt]}}{.45}];
    \draw[very thick] (v1) arc (70:30:4 and 1);
    \draw[very thick]
        plot [smooth, tension=.7]
        coordinates {(v1l) (4,7) (5, 5)(p11)}
        [arrow inside={end={Stealth[width=8pt, length=9pt]}}{.7}];
    % middle path
    \draw[very thick, dashed]
        plot [smooth, tension=.7]
        coordinates {(p11) (-1,-3) (-3,0) (-6, -4) (-11, -2) (p22)};
    %path cP_2
    \draw[very thick, red]
        plot [smooth, tension=.7]
        coordinates {(p22) (-6,4) (v2)}
        [arrow inside={end={Stealth[width=8pt, length=9pt]}}{.65}];
    \draw[very thick, red, -{Rays}]
        (v2) arc (-130:-200:4 and 1) -- (140:10.3);
\end{tikzpicture}%
\hspace{5mm}%
\begin{tikzpicture}[scale=0.25,
    dot/.style={
        draw,circle,minimum size=1mm,inner sep=0pt,outer
        sep=0pt,fill=black},
    filling/.style={gray, opacity=.3}
    ]
    \coordinate          (c) at (0,1); % center of ellipse
    \coordinate [dot]  (p12) at (  90:10)
                        node at  (p12) [above] {$p^{1}_2$};
    \coordinate [dot]  (p22) at ( 160:10)
                        node at (p22) [above left] {$p^{2}_2$};
    \coordinate [dot]   (v1) at ($(c)+(30:4 and 1)$)
                        node at  (v1) [above right] {$v_1$};
    \coordinate        (v1l) at ($(c)+(-43:4 and 1)$); %left E after (v1)
    \coordinate [dot]   (v2) at ($(c)+(-200:4 and 1)$)
                        node at (v2) [below left] {$v_{2}$};
    \coordinate [dot]  (p11) at (v1l)
                        node at  (p11) [below right] {$p^{1}_1$};
    \coordinate [dot]  (p21) at ($(c)+(-233:4 and 1)$)
                        node at (p21) [above] {$p^{2}_1$};
    % circles and elipse
    \draw (0,0) circle (10);
    \draw (0,0) circle (3);
    \filldraw[filling] (c) circle (4 and 1);
    %path cP_1
    \draw[very thick]
        plot [smooth, tension=.7]
        coordinates {(p12) (1,8) (v1)}
        [arrow inside={end={Stealth[width=8pt, length=9pt]}}{.4}];
    \draw[very thick] (v1) arc (30:-43:4 and 1);
    % middle path
    \draw[very thick, dashed]
        plot [smooth, tension=.7]
        coordinates {(p11) (-1,-3) (-3,0) (-6, -4) (-11, -2) (p22)};
    %path cP_2
    \draw[very thick]
        plot [smooth, tension=1]
        coordinates {(p22) (-6,3) (v2)}
        [arrow inside={end={Stealth[width=8pt, length=9pt]}}{.5}];
    \draw[very thick] (v2) arc (-200:-233:4 and 1);
\end{tikzpicture}
\caption{By \textbf{P3}, an ellipse $E$ is
    used at most twice by entering arms.}
\label{fig:P3_construction}
\end{figure}

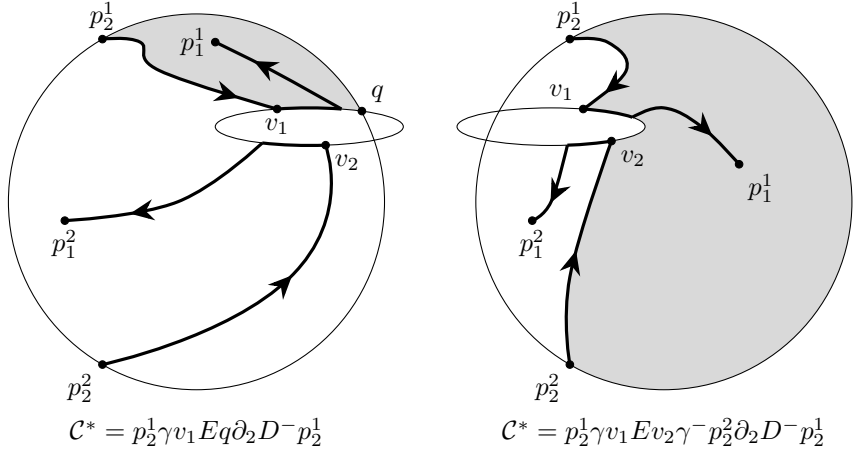
\begin{figure}
\centering
\begin{tikzpicture}[scale=0.25,
    dot/.style={
        draw,circle,minimum size=1mm,inner sep=0pt,outer
        sep=0pt,fill=black}
    ]
    \coordinate          (c) at ( 6, 4); % center of ellipse
    \coordinate [dot]    (q) at ($(c)+(56:5 and 1)$) % cap partial2 D
                        node at  (q) [above right] {$q$};
    \coordinate [dot]  (p12) at ( 120:10)
                        node at  (p12) [above] {$p^{1}_2$};
    \coordinate [dot]  (p22) at (-120:10)
                        node at (p22) [below left] {$p^{2}_2$};
    \coordinate [dot]   (v1) at ($(c)+(110:5 and 1)$)
                        node at  (v1) [below] {$v_1$};
    \coordinate        (v1l) at ($(c)+(70:5 and 1)$); %left E after (v1)
    \coordinate [dot]   (v2) at ($(c)+(-80:5 and 1)$)
                        node at (v2) [below right] {$v_{2}$};
    \coordinate        (v2l) at ($(c)+(-120:5 and 1)$); %left E after (v2)
    \coordinate [dot]  (p11) at (1, 8.5)
                        node at  (p11) [left] {$p^{1}_1$};
    \coordinate [dot]  (p21) at (-7,-1)
                        node at (p21) [below] {$p^{2}_1$};
    % closed curve cC^{\ast}
    \begin{scope}[on background layer]
    \fill[gray, opacity=.3]
        plot [smooth, tension=.7]
        coordinates {(p12) (-3,8.5) (-2,7) (v1)}
        arc (110:56:5 and 1)
        arc (29:120:10) --cycle;
    \node[below] at (0,-11)
        {$\cC^{\ast}=p^{1}_{2} \gamma v_1 E q \partial_2D^{-} p^{1}_2$};
    \end{scope}

    % circle and elipse
    \draw (0,0) circle (10);
    \draw (c) circle (5 and 1);
    %path cP_1
    \draw[very thick]
        plot [smooth, tension=.7]
        coordinates {(p12) (-3,8.5) (-2,7) (v1)}
        [arrow inside={end={Stealth[width=8pt, length=9pt]}}{.8}];
    \draw[very thick] (v1) arc (110:70:5 and 1);
    \draw[very thick]
        plot [smooth, tension=.7]
        coordinates {(v1l) (4, 7) (p11)}
        [arrow inside={end={Stealth[width=8pt, length=9pt]}}{.7}];
    %path cP_2
    \draw[very thick]
        plot [smooth, tension=1]
        coordinates {(p22) (5, -4) (v2)}
        [arrow inside={end={Stealth[width=8pt, length=9pt]}}{.6}];
    \draw[very thick] (v2) arc (-80:-120:5 and 1);
    \draw[very thick]
        plot [smooth, tension=.7]
        coordinates {(v2l) (-1,0) (p21)}
        [arrow inside={end={Stealth[width=8pt, length=9pt]}}{.7}];
\end{tikzpicture}%
\hspace{7mm}%
\begin{tikzpicture}[scale=0.25,
    dot/.style={
        draw,circle,minimum size=1mm,inner sep=0pt,outer
        sep=0pt,fill=black}
    ]
    \coordinate          (c) at (-6, 4); % center of ellipse
    \coordinate [dot]  (p12) at ( 120:10)
                        node at  (p12) [above] {$p^{1}_2$};
    \coordinate [dot]  (p22) at (-120:10)
                        node at (p22) [below left] {$p^{2}_2$};
    \coordinate [dot]   (v1) at ($(c)+(70:5 and 1)$)
                        node at  (v1) [above left] {$v_1$};
    \coordinate        (v1l) at ($(c)+(30:5 and 1)$); %left E after (v1)
    \coordinate [dot]   (v2) at ($(c)+(-50:5 and 1)$)
                        node at (v2) [below right] {$v_{2}$};
    \coordinate        (v2l) at ($(c)+(-80:5 and 1)$); %left E after (v2)
    \coordinate [dot]  (p11) at (4, 2)
                        node at  (p11) [below right] {$p^{1}_1$};
    \coordinate [dot]  (p21) at (-7,-1)
                        node at (p21) [below] {$p^{2}_1$};
    % closed curve cC^{\ast}
    \begin{scope}[on background layer]
    \fill[gray, opacity=.3]
        plot [smooth, tension=.7]
        coordinates {(p12) (-3,8.5) (-2,7) (v1)}
        arc (70:-50:5 and 1)
        plot [smooth, tension=.7]
        coordinates {(v2) (-5, -4) (p22)}
        arc (-120:120:10) --cycle;
    \node[below] at (0,-11)
        {$\cC^{\ast}
        =p^{1}_2 \gamma v_1 E v_{2} \gamma^{-}
         p^{2}_2 \partial_2D^- p^{1}_2$};
    \end{scope}

    % circle and elipse
    \draw (0,0) circle (10);
    \draw (c) circle (5 and 1);
    %path cP_i
    \draw[very thick]
        plot [smooth, tension=.7]
        coordinates {(p12) (-3,8.5) (-2,7) (v1)}
        [arrow inside={end={Stealth[width=8pt, length=9pt]}}{.8}];
    \draw[very thick] (v1) arc (70:30:5 and 1);
    \draw[very thick]
        plot [smooth, tension=.7]
        coordinates {(v1l) (0, 5) (2, 4) (p11)}
        [arrow inside={end={Stealth[width=8pt, length=9pt]}}{.7}];
    %path cP_{i+1}
    \draw[very thick]
        plot [smooth, tension=.7]
        coordinates {(p22) (-5, -4) (v2)}
        [arrow inside={end={Stealth[width=8pt, length=9pt]}}{.5}];
    \draw[very thick] (v2) arc (-50:-80:5 and 1);
    \draw[very thick]
        plot [smooth, tension=.7]
        coordinates {(v2l) (-6,0) (p21)}
        [arrow inside={end={Stealth[width=8pt, length=9pt]}}{.7}];
\end{tikzpicture}
\caption{Proving \textbf{P2} by contradiction.
    If $v_1 E v_{2} \cap \partial_1 D = \varnothing$,
    then $p^{1}_1$ and $p^{2}_1$ are separated by the Jordan curve
    $\cC^\ast$, whose interior is gray.}
\label{fig:P2_construction}
\end{figure}

\subsection{Percolation on IPS}
\label{sub:percolation_on_IPS}

By Proposition~\ref{prop:LR0_measurable}, we know that $\overline{LR}_0=
\cap_{r>0} \overline{LR}_r$ almost surely. In this section we give proofs of
some results regarding percolation on IPS.
Any IPS $\xi$ of intensity $u$ is a countable collection 
of sticks $(\gamma_j)$. Define an equivalence relation by 
saying that two sticks $\gamma$ and $\gamma'$ are
connected if there is a finite sequence of sticks
$\gamma_0 = \gamma, \gamma_1, \ldots, \gamma_n = \gamma'$
such that $\gamma_j \cap \gamma_{j-1} \neq \varnothing$ for $j \le n$.
We call \textit{clusters} the sets obtained by
the union of every stick in the same equivalence class.
One can verify that a.s. every pair of sticks of the soup that intersect
must do so on their interiors. Indeed, this can be seen as a consequence of the
multivariate Mecke equation for Poisson processes (see e.g.
\cite[Eq. (12.1)]{Last_Penrose_2017}), and the fact that measure
$(\mu_2)^2$ on $S^2 \subset \RR^8$ gives measure zero to the set
\begin{equation*}
\Bigl\{(s_1,s_2) \in S^2;\;
\begin{array}{l}
    \text{either $E_0(s_1)$ contains an endpoint of $E_0(s_2)$,}\\
    \text{or $E_0(s_2)$ contains an endpoint of $E_0(s_1)$}
\end{array}\Bigr\}.
\end{equation*}

\begin{lema}
\label{lema:NW_properties}
Let $\xi$ be a IPS of intensity $u > 0$.
\vspace{-2mm}
\begin{enumerate}[(i)]
\item For each $n \geq 1$ consider box
    $B_1 := [-3^{n+1}, 3^{n+1}] \times [3^{n}, 3^{n+1}]$ and rotate it
    around the origin by angles $\frac{\pi}{2}$, $\pi$ and $\frac{3\pi}{2}$,
    obtaining boxes $B_2, B_3$ and $B_4$.
    Let $A_n$ be the event in which each $B_j$ is crossed by some stick in
    the longest direction, forming a circuit of $4$ sticks around the origin.
    Then, we have $\PP(A_n\ \text{i.o.}) = 1$ and $\cV$ does not percolate,
    a.s.
\item Either all clusters of $\xi$ are bounded a.s. 
or there is a unique unbounded cluster that is also dense.
\end{enumerate}
\end{lema}

\begin{proof}
The proof of \textit{(i)} is a consequence of FKG inequality. Using
Lemma~\ref{prop:crossing_box_with_one_stick} we know that
\begin{equation*}
\PP(A_n)
    \geq \PP(LR_1(2 \cdot 3^n; 3))^{4}
    \geq \bigl(1 - \exp[ -c \cdot u ]\bigr)^{4}
    > 0.
\end{equation*}
Moreover, notice that the collection of all four events that
comprise each $A_n$ is made of independent events, since a single
stick cannot cross two of the boxes at the same time. Thus, by Borel-Cantelli
Lemma we conclude that $\PP[A_n\ \text{i.o.}] = 1$.

To prove \textit{(ii)} we use an argument from \cite{nacu2011random}. By
ergodicity of IPS with respect to translations, we know that the event in
which all clusters are bounded has probability 0 or 1. If it is zero, there
is an unbounded cluster and by \textit{(i)} we conclude it must be unique.
To see that this unbounded cluster is also dense, just notice that the
distance between the unbounded cluster and the origin is a finite random
variable $X$ that is also scale-invariant, implying $X=0$ a.s.
\end{proof}

By~\cite{nacu2011random}, we already know that, for
small densities $u$, the carpet is non empty. For completeness, we sketch a
proof adapting the argument
of~\cite[Theorem~1.3]{teixeira_ungaretti2017ellipses}.
In~\cite{ungaretti2017phdthesis} there is a second proof,
based on a theorem of Popov and Vachkovskaia for
multi-scale Poisson stick model~\cite{popov2002note}.

\begin{proof}[Proof of Proposition~\ref{prop:vacant_crossing_boxes_IPS}]
Fix a box $B \subset \RR^2$.
We prove that for $u>0$ sufficiently small there is $\delta(u)> 0$ such that
\begin{equation}
\label{eq:LR0_as_intersection_bounds}
\delta(u)
    \le \PP\bigl( \cap_{r > 0}\overline{LR}_r (B) \bigr)
    \le 1 - \delta(u).
\end{equation}
Joining the estimates in~\eqref{eq:LR0_as_intersection_bounds} with Proposition~\ref{prop:LR0_measurable}, the conclusion follows.

Since our model is already scale-invariant, we can fix $l = 1$.
A standard FKG argument reduces the problem to studying $k = 2$, like in 
\cite[Lemma~7.1]{teixeira_ungaretti2017ellipses}. Consider
$B := [-1,1] \times [0,1]$.
Proposition~\ref{prop:crossing_box_with_one_stick} provides
\begin{equation*}
\PP\bigl( \cap_{r > 0}\overline{LR}_r (B) \bigr)
    \le 1 - \PP(LR_1(2;1/2))
    \le 1 - \delta(u).
\end{equation*}
For the lower bound, we make a coupling with fractal percolation.
For each $n \in \NN$ we partition our original box $B$ into square
boxes of side $2^{-n}$ and denote this family of boxes by
$(B^n_z)_{z \in \Lambda_n}$. Let also $I_n := (2^{-(n+1)}, 2^{-n}]$.
For each $n$ we define the random field
\begin{equation*}
X^n_z
    := \I\{
        \xi_0(s; R \in I_n, \; E_0(s) \cap B^n_z \neq \varnothing ) = 0
        \}
\end{equation*}
and notice that once again $\PP(X_z^n = 1) \geq e^{-cu}$ for some
universal positive constant. Thus, applying the results of Liggett,
Schonmann and Stacey~\cite{liggett1997domination} we obtain for
each $n$ a product random field
$(Y^n_z)_{z \in \Lambda_n}$ that is dominated by
$(X^n_z)_{z \in \Lambda_n}$ and $\PP(Y^n_z = 1) = \beta(u) \to 1$ as
$u \to 0$. The sets
\begin{equation*}
A_0 := B
\quad \text{and} \quad
A_n := A_{n-1} \cap
    \Big(\bigcup_{\substack{\scriptscriptstyle z \in \Lambda_n;
    \\ \scriptscriptstyle Y^n_z = 1}} B^n_z \Big)
\end{equation*}
represent the $n$-th step of construction of fractal percolation model with
parameter $\beta(u)$. Event
\begin{equation*}
\bigl\{\, \xi_0(s; E_0(s) \cap B \neq \varnothing, R > 1) = 0 \, \bigr\}
    \cap \bigl\{ \text{$A_n$ crosses $B$ from left to right} \bigr\}
\end{equation*}
is contained on event $\overline{LR}_{2^{-n}}(B)$ for every $n$.
Applying Theorem 1 of Chayes, Chayes and
Durrett~\cite{ChayesDurrett1988fractal_percolation}, we know that
if we take $u$ sufficiently small then
$\cap_{n \in \NN} \{\text{$A_n$ crosses $B$ from left to right}\}$
happens with positive probability, implying
$\PP(\cap_{r>0} \overline{LR}_r(B)) \geq \delta(u)$ and finishing
the proof of~\eqref{eq:vacant_crossing_bounded_away_IPS}.

By Lemma~\ref{lema:NW_properties} we know $\cV$ a.s. does not percolate.
Finally, we prove that $\cE$ does not percolate for $u \in (0, \bar{u})$.
Notice that with probability at least $\delta(u) > 0$ we
have a vacant circuit around the origin on $B(3l)\backslash B(l)$, using FKG.
Define event $\tilde{A}_n$ in which there is such a circuit on
$B(3^{n})\backslash B(3^{n-1})$. By Fatou's lemma, we have
$\PP(\limsup_n \tilde{A}_n) \geq \limsup_n \PP(\tilde{A}_n) \geq \delta$ and thus
$\PP(\cE \ \text{percolates}) \le 1 - \delta < 1$. Since
IPS is ergodic with respect to translations (see
Remark~\ref{remark:sipss_ergodic}),
hence $\PP(\cE \ \text{percolates}) = 0$.
\end{proof}

\section{H\"older regularity of exploration paths}
\label{sec:Holder_regularity_exploration_paths}

In this section we develop our proof that exploration paths $\gamma_r^0$
satisfy \textbf{H1}, implying this sequence of curves is tight and
providing useful information on their regularity.
Before that, recall that $\tilde{\gamma}_r^0$ is the subpath of $\gamma_r^0$
starting from its last visit to the left side of box $B$.
\begin{prop}
\label{prop:property_varnothing}
Family $\{\tilde{\gamma}^0_r\}_{0 < r \le 1}$ satisfies Property
($\varnothing$) in~\eqref{eq:property_varnothing}, with
constants $\zeta = 2$ and $q, Q$ depending only on $u$.
\end{prop}

\begin{proof}
Let $(B(x_i,r_i); 1 \le i \le n)$ be a family of 2-separated balls with
$x_i \in B$ and fix $r \in (0,1]$. Let $Z$ be the only point in
$\tilde{\gamma}_r^0$ that intersects the left side of $B$ and notice that
$Z$ belongs to at most 1 ball from $(B(x_i, 2r_i))$. Consequently, $Z$ is outside 
at least $n-1$ balls of the collection, say $(B(x_i, 2r_i))_{i \le n-1}$.
For each $i \le n-1$, we define events $A_i$ like in
Lemma~\ref{lema:NW_properties}(i): four congruent rectangles $(B^{i}_j)_{j=1}^4$ are
crossed by a single stick in the longest direction, where
\begin{equation*}
    B^{i}_1 = x_i + [r_i, \sqrt{2} r_i] \times [-\sqrt{2}r_i, \sqrt{2}r_i]
\end{equation*}
and $B^{i}_2,B^{i}_3$ and $B^{i}_4$ are rotations around $x_i$ by angles
$\frac{\pi}{2}, \pi$ and $\frac{3\pi}{2}$, so that
$\cup_{j=1}^4 B^{i}_j \subset D(x_i; r_i, 2 r_i)$. By
Lemma~\ref{prop:crossing_box_with_one_stick},
\begin{equation*}
\PP(A_i)
    \geq \PP\bigl(LR_1\bigl( (\sqrt{2}-1) r_i;\; 4 + 2\sqrt{2}\bigr)\bigr)^{4}
    \geq \bigl(1 - \exp[ -c \cdot u ]\bigr)^{4}
    > 0.
\end{equation*}
Moreover, we modify $A_i$ slightly by requiring for each $B^{i}_j$ that its
crossing by a single stick is achieved by a stick centered inside $B^{i}_j$,
obtaining events $\tilde{A}_i$. This modification is useful because events
$\tilde{A}_i$ are independent and also satisfy $\PP(\tilde{A}_i) \ge c > 0$
for some constant $c = c(u)$, see the proof of the lower bound
in~\cite[Proposition 5.1]{teixeira_ungaretti2017ellipses}.
By construction, on event $\tilde{A}_i$ the 4 sticks that
make the crossings form a circuit inside the annulus $D(x_i; r_i, 2 r_i)$ that
prevents $\tilde{\gamma}^0_r$ to enter $B(x_i,r_i)$. Thus,
\begin{equation*}
\PP(\text{$\tilde{\gamma}^0_r$ intersects $B(x_i, r_i)$ for every $1\le i\le n-1$})
    \le \PP( \cap_{i=1}^{n-1} \tilde{A}_i^{\comp})
    \le (1 - c)^{n-1}.
\end{equation*}
We have an upper bound of the form $Q \cdot q^n$ with $Q = (1-c)^1$ and
$q=1-c$.
\end{proof}

Now, we focus on proving \textbf{H1}. The proof is divided into two
steps. First, we prove bounds on covered 1-arm events for IPS. The second step
is to relate the existence of $k$-arms in a path $\gamma_r^0$ with the
existence of a positive proportion (independent of $r$) of disjoint covered
arms using sticks in $\xi_r$; this allows to use BK inequality to conclude the
result.

\subsection{1-Arm events for IPS}
\label{sub:1_arm_events_for_sipss}

In this subsection, we estimate the probability of covered 1-arm events for
IPS. This bound is new, to the best of our knowledge. As we mentioned in
Section~\ref{sub:previous_results}, reference~\cite[Corollary~8]{nacu2011random}
proves a power law for `vacant' 1-arms in a general scale-invariant
soup of random curves. If $A_\epsilon$ is the event in which there
is a path connecting the inner and outer
boundaries of $D(\epsilon,1)$ that do not cross curves in
$\Gamma_{D(\epsilon,1)}$ for IPS, they prove
\begin{equation}
\label{eq:nacu_poly_bound_recall}
\epsilon^{\bar{\eta}} \le \PP(A_\epsilon) \le k'\epsilon^{\bar{\eta}}
\end{equation}
for positive constants $\bar{\eta}(u)$ and $k'(u)$. In order to obtain
\textbf{H1}, we need similar estimates for covered 1-arms.

\begin{defi}
\label{defi:stick_arm_IPS}
For $0 < l_1 < l_2$, define $\Arm_0(l_1, l_2)$ as the event in which there is a
finite sequence of sticks of $\xi_0$ crossing $D(l_1, l_2)$ in IPS and
let $C_\epsilon := \Arm_0(\epsilon, 1)$. Also, define $\Circ_0(l_1, l_2)$
as the event in which there is a finite sequence
$(s_i)_{i=1}^n \in \xi_0$ such that
$E_0(s_i) \cap E_0(s_{i+1}) = p_i \in D(l_1, l_2)$ for all
$i = 1, \ldots, n$ (with $s_{n+1} = s_1$) and the polygon
$p_1, p_2, \ldots, p_n, p_1$ is a simple curve that contains
$B(l_1)$ in its interior.
\end{defi}

\begin{defi}
\label{defi:stick_arm_events}
For $r > 0$ and $0 < l_1 < l_2$ we define the event
$\Arm_r(l_1, l_2) := \{\cE_r \ \text{crosses} \ D(l_1, l_2)\}$.
Its complement, the event in which there is a vacant circuit in
$D(l_1, l_2)$ with $B(l_1)$ in its interior, is denoted by $\vCirc_r(l_1,l_2)$.
\end{defi}

Notice that scale-invariance implies
$\PP\bigl(\Arm_0(l_1, l_2)\bigr) = \PP\bigl(C_{l_1/l_2}\bigr)$.
One can partially adapt the proof of~\eqref{eq:nacu_poly_bound_recall}
for covered 1-arms, obtaining
\begin{lema}
\label{lema:lbound_power_law_C_epsilon}
Fix $u \in (0, \bar{u})$ and $R > 1$. For any IPS of intensity
$u$ there is $k = k(u,R) > 0$ such that for any
$\epsilon, \epsilon' \in (0, 1/R)$
\begin{equation}
\label{eq:lbound_power_law_C_epsilon}
\PP(C_{\epsilon\epsilon'})
    \geq k \PP(C_\epsilon) \PP(C_{\epsilon'/R}).
\end{equation}

As a consequence, the limit
$\beta(u) := \lim_{\epsilon \to 0}
    \frac{\log \PP(C_\epsilon)}{\log \epsilon}$
exists and belongs to $(0, 2]$. Also, it holds that
$\PP(C_\epsilon) \geq k^{-1}R^\beta \epsilon^\beta$.
\end{lema}

\begin{proof}
Notice that for $\epsilon, \epsilon' \in (0, 1/R)$ one has
$\epsilon\epsilon' < \epsilon < R\epsilon < 1$. Moreover,
$k(u) := \PP[\Circ_0(1, R)] > 0$ and by scale-invariance we have
$k(u) = \PP[\Circ_0(\epsilon, R\epsilon)]$ for any $\epsilon > 0$.
Using FKG inequality we can write
\begin{equation}
\label{eq:lbound_power_law_C_epsilon2}
\PP[C_{\epsilon\epsilon'}]
    \geq \PP[ C_\epsilon,
            \Circ_0(\epsilon, R\epsilon),
            \Arm_0(\epsilon\epsilon', R\epsilon)]
    \geq k \PP[C_\epsilon] \PP[C_{\epsilon'/R}],
\end{equation}
since all events above are increasing. Thus, if we define the function
$g(\epsilon) := kP[C_{\epsilon/R}]$, it follows from
equation~\eqref{eq:lbound_power_law_C_epsilon2} that
\begin{equation*}
g(\epsilon\epsilon')
    = kP[C_{\epsilon\epsilon'/R}] \geq k^2 \PP[C_{\epsilon/R}] \PP[C_{\epsilon'/R}]
    = g(\epsilon) g(\epsilon').
\end{equation*}
Supermultiplicativity implies the existence of the limit
\begin{equation}
\label{eq:tilde_beta_defi}
\tilde{\beta}(u)
    := \lim_{\epsilon \to 0} \frac{\log g(\epsilon)}{\log \epsilon}
    = \sup_{\epsilon < 1/R} \frac{\log g(\epsilon)}{\log \epsilon}
    > 0.
\end{equation}

Existence of the limit characterizing $\beta$ follows from noticing
\begin{equation*}
\beta(u)
    = \lim_{\epsilon \to 0} \frac{\log \PP[C_{\epsilon/R}]}{\log (\epsilon/R)}
    = \lim_{\epsilon \to 0}
        \left[
        \frac{\log g(\epsilon)}{\log \epsilon}
        \frac{\log \epsilon}{\log (\epsilon/R)} -
        \frac{\log k}{\log \epsilon}
        \frac{\log \epsilon}{\log (\epsilon/R)}
        \right]
    = \lim_{\epsilon \to 0} \frac{\log g(\epsilon)}{\log \epsilon}
    = \tilde{\beta}(u).
\end{equation*}

To see that $\beta(u) \le 2$, we apply
Proposition~\ref{prop:crossing_box_with_one_stick}.
Notice that for $\epsilon < 1/2$, crossing with one stick a box whose left
side is a diameter of $B(\epsilon)$ and bottom side has length 1, we have
\begin{equation*}
\PP[C_\epsilon]
    \geq \PP[LR_1(2\epsilon; 1/(2\epsilon))]
    \geq 1 - e^{-uc (2\epsilon)^{2}},
\end{equation*}
implying that
\begin{equation*}
\beta(u)
    = \lim_{\epsilon \to 0} \frac{\log \PP[C_\epsilon]}{\log \epsilon}
    \le \lim_{\epsilon \to 0}
        \frac{\log (1 - \exp[-uc \epsilon^2])}{\log \epsilon}
    = 2.
\end{equation*}

Finally, by \eqref{eq:tilde_beta_defi} we get for
$\epsilon \in (0, 1/R)$ that $\log g(\epsilon) \geq \beta \log \epsilon$ and
hence $f(\epsilon) = k^{-1}g(R\epsilon) \geq k^{-1} R^\beta \epsilon^\beta$.
\end{proof}

However, the upper bound in~\eqref{eq:nacu_poly_bound_recall} is not
easily adaptable. In Proposition~\ref{prop:up_bound_arm_IPS} we show
that for $u \in (0,\bar{u})$ the probability of having a covered 1-arm in IPS
decays polynomially. However, the exponents in our upper and lower bounds do
not coincide, as in~\eqref{eq:nacu_poly_bound_recall}.

A consequence of Definition \ref{defi:stick_arm_events} is that
the function $r \mapsto \PP\bigl(\Arm_r(l_1, l_2)\bigr)$ is decreasing, since by our
standard coupling we have $\cE_{r'} \supset \cE_{r}$ whenever $r' \le r$. Also,
notice that $\Arm_0(l_1, l_2) = \cup_{r>0} \Arm_r(l_1,l_2)$. By this, we have
\begin{equation}
\label{eq:up_bound_arm_truncated_limit}
\PP\bigl(\Arm_r   (l_1, l_2)\bigr)
    \le \lim_{r' \downarrow 0} \PP\bigl(\Arm_{r'}(l_1, l_2)\bigr)
    = \PP\bigl(\Arm_0(l_1,l_2)\bigr).
\end{equation}
Recall that it is important for obtaining \textbf{H1} that the bounds
in~\eqref{eq:hypothesis_H1_IPS} are uniform for the family
$\{\gamma_r^0\}_{0 < r \le 1}$.
By~\eqref{eq:up_bound_arm_truncated_limit} a uniform estimate for covered 1-arm
events is a consequence of
\begin{prop}
\label{prop:up_bound_arm_IPS}
Fix $u \in (0,\bar{u})$. There exists $K(u) > 0$ and $\eta(u) \in (0,1)$
such that for any $0 < l_1 < l_2$ we have
\begin{equation*}
\PP\bigl(\Arm_0(l_1, l_2)\bigr)
    \le K \bigl(l_1/l_2 \bigr)^{\eta}.
\end{equation*}
\end{prop}

\begin{proof}
Using scale-invariance, we can consider only $\Arm_0(1,l)$ for $l > 1$.
Moreover, it is sufficient to prove bounds for the values $l = 2^m$ where
$m \in \NN$, i.e.\ , there are $K(u), \eta(u) > 0$ such that
%Indeed, if Proposition~\ref{prop:up_bound_arm_IPS}
%holds for these values of $l$ then for a generic $l > 1$ we can find
%$m \in \NN$ with $2^{m-1} < l \le 2^m$ and notice
%\begin{equation*}
%\PP\bigl(\Arm_0(1, l)\bigr)
    %\le \PP\bigl(\Arm_0(1, 2^{m-1})\bigr)
    %\le K 2^{-(m-1)\eta}
    %\le (2^\eta K) l^{-\eta}.
%\end{equation*}
%
%Thus, we want to prove there are $K(u), \eta(u) > 0$ such that
\begin{equation}
\label{eq:up_bound_arm0_l}
\PP\bigl(\Arm_0(1, 2^m)\bigr) \le K 2^{-m\eta}
\quad \text{for every $m\in \NN$.}
\end{equation}
Let us partition $\RR^2\backslash \{0\}$ into disjoint annuli using
$A_j := D(2^{j-1}, 2^j)$, where $j \in \ZZ$. Abusing notation a little,
we denote $\Arm_0(2^{j-1}, 2^j)$ by $\Arm_0(A_j)$ and define
\begin{equation*}
\vCirc_0(A_j)
    := \Arm_0(A_j)^{\comp}
    = \smash{\bigcap_{r>0}} \Arm_r(A_j)^{\comp}
    = \smash{\bigcap_{r>0}} \vCirc_r(A_j).
\end{equation*}
Here, we emphasize that in this whole section we do
not rely on the a.s. equality $\overline{LR}_0(B) = \cap_{r>0} \overline{LR}_r(B)$
from Proposition~\ref{prop:LR0_measurable} (actually,
Proposition~\ref{prop:LR0_measurable} a consequence of the results in this
section). By the estimates in~\eqref{eq:LR0_as_intersection_bounds} and FKG
inequality, if we can choose four boxes contained in $A_j$, similarly to
Lemma~\ref{lema:NW_properties}, so that for $u \in (0, \bar{u})$
there is a constant $\delta(u) > 0$ such that
\begin{equation*}
\PP(\vCirc_0(A_j))
    = \PP\Bigl(\smash{\bigcap_{r>0}} \vCirc_r(A_j)\Bigr)
    \ge \delta(u),
    \quad \text{for every $j \in \ZZ$}.
\end{equation*}

Notice that on the event $\Arm_0(1, 2^m)$ none of the events
$\vCirc_0(A_j)$ for ${1 \le j \le m}$ happened. 
If events $\{\vCirc_0(A_j)\}_{j\in \ZZ}$ were independent,
the proof would be already over. However, we know by
Lemma~\ref{lema:decay_of_correlations_HPS} that these events have a
correlation that decays slowly with distance. This decay is so slow that
even the techniques employed in Lemma 4.7 of~\cite{ahlberg2018sharpness}
are not enough. To illustrate, we pick some fixed subset of indices
$N \subset [m]$ and use the bound $\PP\bigl(\Arm_0(1, 2^m)\bigr)
    \le \PP\bigl(\vCirc_0(A_j)^{\comp}, \ \text{for} \ j \in N\bigr)$
trying to take advantage of the decay of correlations. The best bound we
can obtain is $K 2^{- \eta \sqrt{m}}$ which is weaker than the one in
Proposition~\ref{prop:up_bound_arm_IPS}. To get stronger bounds, we
consider random sets of indices.

For any set $B \subset \RR^2$ we can consider the restriction of
$\xi_0$ in which we only look for sticks intersecting the set $B$.
In this case, we denote the covered set observed in this restriction by
$\cE_0(B)$. Define the random variables
\begin{equation}
\label{eq:def_Drj}
D_j
    := \max \{n \in \ZZ;\;
        \text{$n < j$ and $\cE_0(A_j) \cap A_n = \varnothing$}\}.
\end{equation}

In words, $D_j$ searches the index of the largest annulus smaller than
$A_j$ that was not intersected by sticks touching $A_j$. Since
$\PP\bigl(0 \in \cE_0\bigr) = 0$, we have $D_j$ are always finite.
Scale-invariance implies that $j-D_j$ are identically distributed,
but clearly they are not independent. As a first computation, notice that
\begin{equation*}
\PP\bigl(D_m \le 0\bigr)
    \le \PP\bigl(\exists s \in \xi_0;
        E_0(s) \cap B(1) \neq \varnothing, R \geq (2^m - 1)/2\bigr)
    \le u c 2^{-m}
\end{equation*}
for some constant $c$ independent of $u$.
Notice that this event has the decay we want.
On the other hand, on the event $\{D_m \geq 1\}$ we have
\begin{equation*}
\PP\bigl(\Arm_0(1, 2^m), D_m \geq 1\bigr)
    \le \PP\bigl(\vCirc_0(A_{D_m})^{\comp}, D_m \geq 1\bigr)
    \le 1 - \delta.
\end{equation*}
The idea is then to iterate this reasoning, discovering step by step how many
annuli we should discard in order to try again finding a new vacant circuit.
From now on we fix $m \in \NN$ and consider the sequence of random variables
$(I_{j})_{j \geq 0}$ with $I_{0} = m$ and $I_{j} = D_{I_{j-1}}$.
Sequence $I_{j}$ is clearly decreasing and is illustrated on
Figure~\ref{fig:successive_invasions}. It holds
\begin{lema}
\label{lema:up_bound_cap_random_vacant_circuits}
For every $t \in \NN$ one has
$\PP\bigl(\cap_{j=1}^{t} \vCirc_0(A_{I_{j}})^{\comp}\bigr) \le (1 - \delta)^t$.
\end{lema}

\begin{proof}
Let $\cF_0$ be the smallest $\sigma$-algebra that makes $\xi_0(A_{I_{0}})$
measurable. Inductively, for every $t \in \NN$ we can define $\cF_t$, the
smallest $\sigma$-algebra that makes $\xi_0(A_{I_{j}})$ for
$j = 1, \ldots, t$ all measurable. Notice that $I_{t}$ is
$\cF_{t-1}$ measurable. Also, on $\cF_{t-1}$ when $I_{t} = l$ and
$I_{t-1} = n$ the only sticks intersecting $A_{l}$ that were not explored
yet are the ones in
\begin{equation}
\label{eq:defi_Gamma_ln}
\Gamma(l,n)
    := \{s; E_0(s) \cap A_{l} \neq \varnothing,
        E_0(s) \cap A_{n} = \varnothing\},
\end{equation}
since the sticks that intersect $A_{n}$ have not intersected $A_l$. Denoting by
$\xi_{|\Gamma(l,n)}$ the restriction of $\xi$ to $\Gamma(l,n)$, this means that
on event $\{I_{t} = l, I_{t-1} = n\}$ we have
\begin{align*}
\PP\bigl(\vCirc_0(A_{I_{t}})^{\comp} \, \mid \cF_{t-1}\bigr)
    &= \PP\bigl(\vCirc_0(A_{l})^{\comp} \
        \text{on $\xi_{|\Gamma(l,n)}$} \mid \cF_{t-1}\bigr)
    = \PP\bigl(\vCirc_0(A_{l})^{\comp} \
        \text{on $\xi_{|\Gamma(l,n)}$}\bigr) \\
    &\le \PP\bigl(\vCirc_0(A_{l})^{\comp} \bigr)
    \le 1 - \delta.
\end{align*}
On the second equality we used independence of PPP $\xi$ on disjoint
regions and on the second line we use that event
$\vCirc_0(A_{I_{t}})^{\comp}$ is increasing.
Since the bound above holds for every possible value of $l$ and $n$
we have for any $t$ that
\begin{align*}
\PP\Bigl(\bigcap_{j=1}^{t} \vCirc_0(A_{I_{j}})^{\comp}\Bigr)
    &= \EE\Bigl[
        \I\bigl\{\bigcap_{j=1}^{t-1} \vCirc_0(A_{I_{j}})^{\comp}\bigr\} \cdot
    \PP\bigl(\vCirc_0(A_{I_{t}})^{\comp} \, | \cF_{t-1}\bigr)
    \Bigr]
    \le (1 - \delta) \cdot
        \PP\Bigl(\bigcap_{j=1}^{t-1} \vCirc_0(A_{I_{j}})^{\comp}\Bigr).
\end{align*}
Iterating $t$ times, we are done.
\end{proof}

Now, we consider the possibilities for how many tries we have before
$I_{t} < 1$. Define
\begin{equation*}
L_{j}
    := I_{j-1} - I_{j} \ \text{for every} \ j \geq 1,
\end{equation*}
the random variable that counts the amount of rings discarded moving from step
$j-1$ to step $j$. Notice that knowing the family $(L_{j})$ is equivalent to
knowing $(I_{j})$. The coordinates of these processes are not independent.
Let us denote by $Y$ the distribution of $L_{1}$, i.e.,
\begin{equation*}
Y
    \distr j -
    \max \{n; n < j \ \text{and} \ \cE_0(A_j) \cap A_n = \varnothing\}.
\end{equation*}
Just as we saw in the proof of
Lemma~\ref{lema:up_bound_cap_random_vacant_circuits}, given the values of
$L_{1}, \ldots, L_{j}$ the random variable $L_{j+1}$ should be
stochastically smaller than $Y$, since we have the information that sticks from
$A_{I_{j}}$ did not intersect $A_{I_{j+1}}$.
This is made precise the following result.
\begin{lema}
\label{lema:stochastic_domination_for_L}
Let $(Y_{j})_{j \geq 1}$ be a family of iid. random variables distributed as
$Y$. For any $t \in \NN$ we have that
\begin{equation}
\label{eq:stochastic_domination_for_L}
\smash{\sum_{j=1}^{t} Y_{j}}
\quad \text{dominates stochastically} \quad
\smash{\sum_{j=1}^{t} L_{j}}.
\end{equation}  
\end{lema}
\begin{proof}
\vspace{2mm}
By adding some extra randomness, we define an explicit sequence of random
variables $Y_{j}$ that verifies~\eqref{eq:stochastic_domination_for_L}.
Just like we defined $L_{j} = I_{j-1} - I_{j}$ it is convenient to define
our sequence $Y_{j}$ as the difference of $\tI_{j-1} - \tI_{j}$ for a
decreasing sequence of random variables that we define below. Proving the
stochastic domination is then the same as showing that almost surely
$\tI_{t} \le I_{t}$ for every $t$.

Define $\tI_{j} := I_{j}$ for $j = 0, 1$ (which implies
$Y_{1} = L_{1}$). We denote by $\xi_{|l}^{n}$ the restriction of
PPP $\xi$ to sticks that intersect $A_l$ but not $A_n$ and by $\xi_{|l,n}$
the restriction to sticks that intersect both $A_l$ and $A_n$.
Consider a family $(\tilde{\xi}_{|l,n};\; l, n \in \ZZ)$ of
independent realizations of $\xi_{|l,n}$.

By definition of $I_{j}$, we have that
\begin{align*}
I_{j+1}
    &:= \max\{l;\; \cE_0(A_{I_{j}}) \cap A_l = \varnothing \}
    = \max\{l;\; \cE_0(A_{I_{j}}) \cap A_l = \varnothing \ 
    \text{on $\smash{\xi_{|I_{j}}^{I_{j-1}}}$}\},
\end{align*}
since no stick from $A_{I_{j-1}}$ can intersect $A_{I_{j}}$. We define
$\tI_{j+1}$ by setting
\begin{equation*}
\tI_{j+1}
    := \max\{l;\; \cE_0(A_{\tI_{j}}) \cap A_l = \varnothing \ 
    \text{on $\xi_{|\tI_{j}}^{\tI_{j-1}} \cup
    \tilde{\xi}_{|\tI_{j}, \tI_{j-1}}$}\}.
\end{equation*}
By adding this independent realization, we make
$Y_{j} = \tI_{j+1} - \tI_{j}$ with the same distribution as
$Y$, and $Y_{j}$ is by construction an independent sequence. Moreover, we
claim that $\tI_{j} \le I_{j}$ for every $j$. Indeed, for $j=0,1$ we
actually have an equality and if we assume that $\tI_{j} \le I_{j}$ then it
is not possible for $I_{j+1} < \tI_{j+1}$ since this would imply there is
some stick from $\xi_{|I_{j}}^{I_{j-1}}$ that intersects both $A_{I_{j}}$
and $A_{I_{j+1}+1}$. This stick must also intersect $A_{\tI_{j}}$,
implying that $\tI_{j+1} \le I_{j+1}$, a contradiction
(see Figure~\ref{fig:successive_invasions}).
\end{proof}

\begin{figure}[ht]
\centering
\begin{tikzpicture}[scale = .7,
  used/.style={fill=gray!20},
  tIused/.style={fill=gray!60},
  discarded/.style={fill=white}
  ]

%text
\node[right] at (11.4, 3.5) {$A_{I_{0}}$};
\node[right] at ( 6.1, 3.5) {$A_{I_{1}}$};
\node[right] at ( 3.8, 3.5) {$A_{I_{2}}$};
\node[right] at (11.4,-3.5) {$A_{\tI_{0}}$};
\node[right] at ( 6.1,-3.5) {$A_{\tI_{1}}$};
\node[right] at ( 2.4,-3.5) {$A_{\tI_{2}}$};

\begin{scope}
\clip (2,-3) rectangle (13,3);
%coloring the used rings
\foreach \x/\y in {
    13/used, 12/discarded, 11/discarded, 10/discarded, 9/discarded,
    8/used, 7/discarded, 6/used,5/tIused, 4/discarded, 3/used,
    2/discarded, 1/discarded}{
  \draw[\y] (0,0) circle (\x);
  %\node[right] at (\x, 0) {$\x$};
}

%some sticks
\draw[rotate around={ 30:(7.3, .4)}] (7.3, .4) circle (1 and 0);
\filldraw (7.3, .4) circle (.05);
\draw[rotate around={ 45:(3, 2.5)}] (3, 2.5) circle (.5 and 0);
\filldraw (3, 2.5) circle (.05);
\draw[rotate around={ 105:(2.3, -2.3)}] (2.3, -2.3) circle (.3 and 0);
\filldraw (2.3, -2.3) circle (.05);
\draw[rotate around={-15:(13.5,-2)}] (13.5,-2) circle (5.5 and 0);
\draw[rotate around={ 95:(8.3, .7)}] (8.3, .7) circle (2 and 0);
\filldraw (8.3, .7) circle (.05);
% thick sticks
\draw[rotate around={-10:(9.3, -2)}, very thick]
    (9.3, -2) circle (4 and 0);
\filldraw (9.3, -2) circle (.05);
\draw[rotate around={15:(11.3, 1)}, very thick]
    (11.3, 1) circle (4 and 0);
\filldraw (11.3, 1) circle (.05);
\end{scope}

\end{tikzpicture}
    \caption{Annuli from sequence $A_{I_j}$ (light gray) represent
    successive new tries for vacant circuits. Sticks from $A_{I_{j-1}}$ do not
    intersect $A_{I_j}$, which prevents $(L_j)$ from being iid. On
    Lemma~\ref{lema:stochastic_domination_for_L} we add independent sticks
    (thick lines) to build iid. sequence $(Y_j)$. The picture emphasizes
    $A_{\tI_2}$ (dark gray) in a realization in which $\tI_2 < I_2$.}
\label{fig:successive_invasions}
\end{figure}
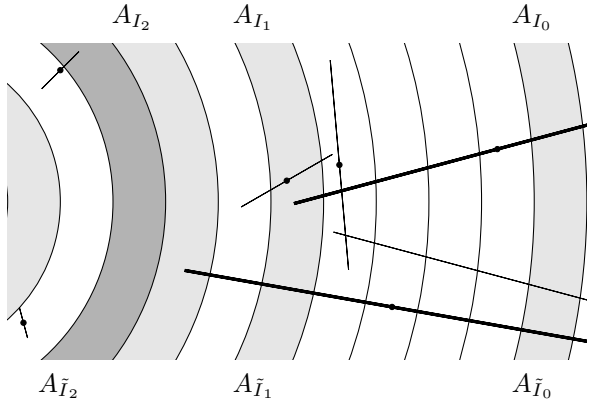
Moving on with the proof of Proposition \ref{prop:up_bound_arm_IPS}, we split
the event $\Arm_0(1, 2^m)$ into two events according to whether we were able to
find a sufficient amount of tries for vacant circuits or not. Let
\begin{equation}
\label{eq:number_of_invasions}
T
    := \max \{ j; I_{j} \geq 1\}
    = \max \{ j; L_{1} + \ldots + L_{j} \le m-1 \}.
\end{equation}

In the next lemma, we prove that with high probability we have a linear number
of attempts:

\begin{lema}
\label{lema:exponential_Markov_for_Lj}
For any $u \in (0,\bar{u})$ there exists positive constants
$\kappa = \kappa(u)$ and $c = c(u)$ such that
\begin{equation*}
\PP\Bigl( \sum_{j=1}^{\lfloor \kappa m \rfloor} L_{j} \geq m\Bigr)
    \le e^{- c m}.
\end{equation*}
\end{lema}

\begin{proof}
Using Lemma~\ref{lema:stochastic_domination_for_L}, we can write for any
$t \in \NN$ that
$
\PP\Bigl( \smash{\sum_{j=1}^{t}} L_{j} \geq m\Bigr)
    \le \PP\Bigl( \smash{\sum_{j=1}^{t}} Y_{j} \geq m\Bigr).
$
Fix some $\theta \in (0, \log 2)$. We have for $n \geq 3$ that
\begin{align*}
\PP(Y \geq n)
  &= \PP(\exists s \in \supp \xi_0;
    E_0(s) \cap A_1 \neq \varnothing, \,
    E_0(s) \cap A_{2-n} \neq \varnothing) \\
  &= \PP(\exists s \in \supp \xi_0;
    E_0(s) \cap \partial B(1) \neq \varnothing, \,
    E_0(s) \cap \partial B(2^{2-n}) \neq \varnothing) \\
  &= 1 -
    \exp[ -u
      \mu_2(s; E_0(s) \cap \partial B(1) \neq \varnothing, \,
          E_0(s) \cap \partial B(2^{2-n} ) \neq \varnothing)
    ] \\
  &\le 1 - \exp[-u c 2^{-n}] \le u c 2^{-n}.
\end{align*}
where the inequality follows from
equation~\eqref{eq:crossing_annulus_simpler} in
Lemma~\ref{lema:crossing_annulus}. This implies
\begin{align}
\EE[e^{\theta Y}]
  &= \sum_{n \geq 2} e^{\theta n} \PP(Y = n)
    \le e^{2\theta} + \sum_{n\geq 3} e^{\theta n} \PP(Y \geq n)
    \le e^{2 \theta} + \sum_{n \geq 3} u c e^{-(\log 2 - \theta)n} \nonumber\\
  &\le e^{2\theta} + u c_\theta < \infty,
\label{eq:exp_moment_for_Y}
\end{align}
for some constant $c_\theta > 0$. The exponential moment
in~\eqref{eq:exp_moment_for_Y} and Markov inequality imply
\begin{equation*}
\PP\Bigl( \smash{\sum_{j=1}^{t}} Y_{j} \geq m\Bigr)
    \le e^{ -\theta m } \cdot \EE\bigl[e^{\theta \sum_{j=1}^t Y_{j}}\bigr]
    = \exp\bigl[\; - \theta m + t \log \EE[e^{\theta Y}] \; \bigr].
\end{equation*}

Take $t = \lfloor\kappa m \rfloor$ where
$0 < \kappa < \frac{\theta}{\log (e^{2\theta} + c_\theta u)}$. We obtain
\begin{equation*}
\PP\Bigl( \smash{\sum_{j=1}^{t}} Y_{j} \geq m\Bigr)
    \le \exp\bigl[- m(\theta - \kappa \log (e^{2\theta} + c_\theta u))\bigr].
    \qedhere
\end{equation*}
\end{proof}

Split event $\Arm_0(1, 2^{m})$ according to whether
$T < \lfloor \kappa m \rfloor$ or not. Notice that
$\{T < \lfloor\kappa m \rfloor\}
    = \big\{ L_{1} + \ldots + L_{\lfloor \kappa m \rfloor} \geq m \big\}$.
Hence, we have by Lemma~\ref{lema:exponential_Markov_for_Lj}
\begin{align*}
\PP\bigl(\Arm_0(1, 2^m)\bigr)
  &\le \PP\bigl(T < \lfloor \kappa m \rfloor\bigr) +
        \PP\bigl(\Arm_0(1, 2^m ), T \geq \lfloor \kappa m \rfloor\bigr) \\
  &\le \exp[- c m] +
  \PP\bigl(
    \cap_{j=1}^{\lfloor \kappa m \rfloor} \vCirc_0(A_{I_{j}})^{\comp} \,
  \bigr) \\
  &\le \exp[-cm] + (1 - \delta)^{\kappa m},
\end{align*}
by Lemma~\ref{lema:up_bound_cap_random_vacant_circuits}. This means that
in~\eqref{eq:up_bound_arm0_l} we can take $K := 2$ and
\begin{equation*}
\eta
    := \frac{c(u) \wedge (\kappa \log(1 - \delta)^{-1})}{\log 2}
    \le \frac{c(u)}{\log 2}
    < \frac{\theta}{\log 2}
    < 1. \qedhere
\end{equation*}
\end{proof}

\subsection{BK inequality}
\label{sub:bk_inequality}

In this section we extend the covered 1-arm bound we obtained for IPS in
Proposition~\ref{prop:up_bound_arm_IPS} to a bound involving more than
one arm. The canonical way to make this passage is to apply the so-called
BK inequality~\cite{van1985inequalities}, which gives an upper bound on
the \textit{disjoint occurrence} of events, as we describe below.

Since we are dealing with a continuum percolation model,
we rely on an extension of BK inequality for marked Poisson point
processes~\cite{van1996note}. So, for this part we see IPS as
a marked PPP on $\RR^{2}$ that gives the centers of sticks, with marks
giving direction and major axis length. A realization $\xi$ of the process can
be seen as a collection of pairs
\begin{equation*}
\bigl\{
    s_i = (z_i, m_i);\;
    z_i \in \RR^{2}, m_i \in \RR_{+} \times (-\pi/2, \pi/2], i \in \NN
    \bigr\}.
\end{equation*}
For any region $U \subset \RR^{2}$ we define 
\begin{equation*}
\xi_{\Vert U} := \{s_i \in \xi; \ z_i \in U\}
\quad \text{and} \quad
\xi_{U} := \{s_i \in \xi; \ E(s_i) \cap U \neq \varnothing\},
\end{equation*}
the restriction of $\xi$ to sticks centered in $U$ and intersecting $U$,
respectively. We also define
\begin{equation*}
[\xi_{\Vert U}] := \{\xi'; \ \xi'_{\Vert U} = \xi_{\Vert U}\}
\quad \text{and} \quad
[\xi_{U}] := \{\xi'; \ \xi'_{U} = \xi_{U}\},
\end{equation*}
the collections of all realizations that coincide with $\xi_{\Vert U}$ and
$\xi_{U}$ in $U$. Reference~\cite{van1996note} contains the following definitions.
\begin{defi}
\label{defi:centered_events}
We say that an event $A$ is \textit{centered} on $U \subset \RR^{2}$ if
when $\xi \in A$ we have also $[\xi_{\Vert U}] \subset A$, i.e.,
we have $\I_{A}(\xi) = \I_{A}(\xi_{\Vert U})$.
We say \textit{$A$ depends only on $U$} if $\xi \in A$ implies
$[\xi_{U}] \subset A$.
\end{defi}
\begin{defi}
\label{defi:disjoint_occurrence}
We say that events $A$ and $B$ \textit{occur disjointly} if
\begin{equation}
A \circ B := \left\{ \omega;\
 \begin{array}{l}
   \text{there are disjoint regions $U$ and $W$ of $\RR^{2}$} \\
   \text{such that $[\omega_{\Vert U}] \subset A$ and $[\omega_{\Vert W}] \subset B$}
 \end{array}
\right\}.
\end{equation}
\end{defi}
\begin{remark}
In \cite{van1996note}, instead of `$A$ being centered on $U$' the authors use
that `$A$ lives on $U$'. However, this name can be misleading since
sticks centered outside region $U \subset \RR^2$ may still intersect $U$.
Moreover, for ensuring measurability, we restrict regions $U$ and $W$
above to the (countable) family of finite unions of rectangles with rational coordinates. 
\end{remark}

Recall that an event $A$ is \textit{increasing} if $\xi \in A$ and
$\xi \subset \xi'$ then $\xi' \in A$. The following is a slight improvement of 
the theorem from \cite{van1996note}.
\begin{lema}
\label{lema:bk_ineq}
Let $V \subset \RR^2$ be a bounded region and consider two increasing events
$A$ and $B$ that depend only on sticks intersecting $V$. Then,
\begin{equation}
\label{eq:BK_inequality_mPPP}
\PP(A \circ B) \le \PP(A) \PP(B).
\end{equation}
\end{lema}

\begin{proof}
The original theorem from~\cite{van1996note} is for events centered on $V$.
Since $V$ is bounded we have $V \subset B(n)$ for any large euclidean ball,
say $n \geq n_0$. Let us define
\begin{equation*}
\xi_{V \Vert B(n)}
    := \{s_i = (z_i, m_i) \in \xi;
        \ E(s_i) \cap V \neq \varnothing,\ z_i \in B(n)\}.
\end{equation*}
and for $n \geq n_0$ let $A_n := \{\xi;\ \xi_{V\Vert B(n)} \in A\}$.
Each $A_n$ is an increasing event that is centered on $B(n)$.
Moreover, events $A_n$ are nested and satisfy $A = \cup_n A_n$.
Define events $B_n$ in an analogous way. Notice that
\begin{equation}
\label{eq:A_circ_B_subset_An_circ_Bn}
A \circ B
    \subset (A_n \circ B_n) \cup
        \{\exists s_i \in \xi;\ E(s_i) \cap V \neq \varnothing,
        \ z_i \notin B(n) \}.
\end{equation}
By union bound and the BK inequality from~\cite{van1996note} we can write
\begin{equation*}
\PP(A \circ B)
    \le \PP(A_n)\PP(B_n) +
        \PP(s_i \in \xi;\ E(s_i) \cap V \neq \varnothing,
            \ z_i \notin B(n))
\end{equation*}
and it suffices to show the second term tends to zero as $n \to \infty$.
By~\eqref{eq:crossing_annulus_simpler} on Lemma~\ref{lema:crossing_annulus}, we have
\begin{equation*}
\PP(s_i \in \xi;\ E(s_i) \cap V \neq \varnothing, \ z_i \notin B(n))
    \le u c \Bigl( \frac{n}{n_0} - 1 \Bigr)^{- 1}
    \xrightarrow{n \to \infty} 0.\qedhere
\end{equation*}
\end{proof}

Notice that $\Arm_r(l_1, l_2)$ is an event that depends only on
sticks touching $B(l_2)$. Moreover, if events $A_i$ are increasing
then $\bigcirc_{i=1}^k A_i$ is also increasing. Iterating
Lemma~\ref{lema:bk_ineq} we obtain
\begin{coro}
\label{coro:bk_ineq_for_arm_events}
Let $r > 0$ and $l_2 > l_1 > 0$. For any $k \in \NN$ we have
\begin{equation*}
\PP(\bigcirc_{i=1}^{k} \Arm_r(l_1, l_2))
    \le \PP(\Arm_r(l_1, l_2))^k.
\end{equation*}
\end{coro}

\subsection{Proof of Hypothesis \textbf{H1}}
\label{sub:proof_of_hypothesis_h1}

We now have all the tools we need to prove hypothesis \textbf{H1} for our
sequence of exploration paths.
\begin{defi}
For any $r>0$ and $l\geq 0$, define the $k$-arm events for
exploration path $\gamma_r^0$ as
\begin{equation*}
k\text{-}\EArm_{r}(z; l_1, l_2)
    = \{D(z; l_1, l_2) \
    \text{is traversed by $k$ separate segments of $\gamma_r^0$} \}.
\end{equation*}
\end{defi}

We want to prove for all $k \in \NN$ and $0 < l_1 < l_2$
we have uniformly in $r > 0$ and $z \in \RR^2$ that
\begin{equation}
\label{eq:hypothesis_H1_IPS_recall}%
\PP(k\text{-}\EArm_{r}(z; l_1, l_2))
    \le K_k \cdot \Bigl(\frac{l_1}{l_2}\Bigr)^{c(u)k}
\end{equation}
for positive constants $K_k(u) < \infty$ and $c(u) > 0$.

Firstly, notice that the existence of $k$ exploration arms in $D(z; l_1, l_2)$
implies the existence of at least $\lfloor k/2 \rfloor$ entering exploration
arms. If they did not share any sticks, the proof would be a trivial
application of BK inequality, but \textbf{P3} from
Lemma~\ref{lema:topological_properties} shows that some sticks can be shared.
Fortunately, by Proposition~\ref{prop:sticks_cap_boundary_ball} we know that
\begin{equation*}
\mu_2(s; \# E_0(s) \cap \partial B(l) = 2) = 2\pi < \infty,
\end{equation*}
that is, the measure of all sticks (in full IPS) that intersect the boundary
of $B(l)$ twice is finite. 

\begin{proof}[Proof of Theorem \ref{teo:hypothesis_H1_IPS}]
To prove~\eqref{eq:hypothesis_H1_IPS_recall} we notice that by
scale-invariance
\begin{equation}
\label{eq:exparm_r_to_exparm_1_relation}
\PP(k\text{-}\EArm_r(l_1,l_2))
    = \PP\bigl(k\text{-} \EArm_1( \tfrac{l_1}{r}, \tfrac{l_2}{r})\bigr)
    = \PP\bigl(k\text{-}\EArm_{\frac{r}{l_1}}(1,\tfrac{l_2}{l_1})\bigr),
\end{equation}
so it suffices to prove that there is a constant $c(u) > 0$ so that
for any $m \in \NN$
\begin{equation}
    \sup_{r \in (0,1]}\PP(k\text{-}\EArm_{r}(1,2^m)) \le K_k e^{- c m k}.
\end{equation}

Consider the circles $\partial B(2^j)$ for $j = 1, \ldots, m$ and let $N_r(j)$
denote the number of sticks with radius greater than $r$ that intersect
$\partial\! B(2^j)$ twice.\! By our standard coupling, we have that
$(N_r(1),\! \ldots, N_r(m))$ is stochastically smaller than
$(N_0(1), \ldots, N_0(m))$.  Moreover, by
Proposition~\ref{prop:sticks_cap_boundary_ball} we
know $N_0(j)$ has distribution $\Poi(2\pi u)$ for every $j$.
Define
\begin{equation*}
J = J_{r,k}
    := \min\bigl\{j \in \NN; N_r(j) \le \tfrac{k}{4}\bigr\},
\end{equation*}
with $\min \varnothing = \infty$. On event $\{k\text{-}\EArm_{r}(1,2^m), J = j\}$
with $j < m$ we have also
$\{k\text{-}\EArm_{r}(2^j,2^m),\break N_r(j) \le \frac{k}{4}\}$.
Notice that on this event we can deduce that there are a lot of disjoint
arms:

\begin{itemize}
\item We have at least $\lfloor \frac{k}{2} \rfloor$ entering exploration arms.
    By \textbf{P3} in Lemma~\ref{lema:topological_properties}, each stick
    intersecting $\partial B(2^j)$ can participate in at most 2 entering
    exploration arms. If for each stick that participates in two entering
    exploration arms we discard one of the entering arms arbitrarily, we
    discard at most $k/4$ entering arms in total, so we have at least
    $\lfloor \frac{k}{2} \rfloor - \frac{k}{4} > \frac{k}{4} - 1$
    entering arms from $\partial B(2^m)$ to $\partial B(2^j)$ that do not
    share any sticks.

\item An entering arm may walk over the bottom side of our box $B$ and such an
    arm does not make a covered 1-arm of sticks connecting the internal and
    external boundaries of $D(z;2^j, 2^m)$. However, the same reasoning in
    \textbf{P1} in Lemma~\ref{lema:topological_properties}
    shows us that at most one entering arm of $D(z;2^j, 2^m)$ can use the
    boundary of $B$. Thus, we have at least $\frac{k}{4} - 2$ disjoint
    stick arms in $D(2^j, 2^m)$.
\end{itemize}

By this, we can conclude using BK inequality in
Corollary~\ref{coro:bk_ineq_for_arm_events} that for $k \geq 6$ one has
\begin{align*}
\PP&\bigl( k\text{-}\EArm_{r}(1,2^m), J \le \lfloor m/2 \rfloor \bigr) \\
    &= \sum_{j=1}^{\lfloor m/2 \rfloor}
        \PP\bigl(k\text{-}\EArm_{r}(1,2^m), J = j\bigr)
    \le \sum_{j=1}^{\lfloor m/2 \rfloor}
        \PP\Bigl( 
        \bigcirc_{i=1}^{\bigl\lceil \tfrac{k}{4} - 2\bigr\rceil} \Arm_{r}(2^j,2^m) 
        \Bigr) \\
    &\le \Bigl\lfloor\frac{m}{2}\Bigr\rfloor \!\cdot 
        \PP\bigl(
        \Arm_r( 2^{\lfloor m/2\rfloor}, 2^m)
        \bigr)^{\bigl\lceil \tfrac{k}{4} - 2\bigr\rceil} 
    \le \Bigl\lfloor\frac{m}{2}\Bigr\rfloor
        \bigl( K e^{- c m} \bigr)^{\bigl\lceil \tfrac{k}{4} - 2\bigr\rceil}
    \le K_k \cdot e^{-c m k},
\end{align*}
where the penultimate inequality comes from
Proposition~\ref{prop:up_bound_arm_IPS}. On the other hand, we have
\begin{equation*}
\PP\bigl(k\text{-}\EArm_{r}(1,2^m), J > \lfloor m/2 \rfloor \bigr)
    \le \PP\Bigl(\bigcap_{j \in [m/2]} N_r(j) \geq k/4 \Bigr)
    \le \PP\Bigl(\bigcap_{j \in [m/2]} N_0(j) \geq k/4\Bigr)
\end{equation*}
by the stochastic domination we just described. To finish the proof it
suffices to show
\begin{lema}
\label{lema:hypothesis_H1_part2}
There is $c = c(u) > 0$ such that for every $k\in \NN$ there is
$K_k = K_k(u) > 0$ with
\begin{equation}
\label{eq:hypothesis_H1_part2}
\PP\Bigl( \smash{\bigcap_{j \in [m]}} N_0(j) \geq k \Bigr)
    \le K_k \cdot e^{-c mk}, \quad \text{for every $m \geq 1$.}
\end{equation}
\end{lema}

\begin{proof}
Let us define
$\Lambda_m
:= \bigl\{s;\; \text{$\# E_0(s) \cap \partial B(2^j)=2$ for some
$j \in [m]$}\bigr\}$ and $\lambda_m := \mu_2(\Lambda_m)$.
By union bound, we know that $\lambda_m \le 2\pi m$. Also, define
\begin{align*}
\Lambda_m(i,j)
  &:= \bigl\{s; \# E_0(s) \cap \partial B(2^l) = 2 \ \text{iff}
    \ l \in [i,j]\bigr\}.
\end{align*}
and notice that the sets $\Lambda_m(i,j)$ with $1 \le i \le j \le m$
form a partition of $\Lambda_m$.

The number of sticks in $\supp \xi_0 \cap \Lambda_m$ is a random variable $M$
with distribution $\Poi(u\lambda_m)$. After checking the number of sticks
in $\Lambda_m$, we want to know for each stick $s_l$  the number circles
$\partial B(2^j)$ with $j \in [m]$ it has crossed twice, which we denote by
$X_l$. Since we are working with a PPP, each $s_l$ has, independently of the
others, probability $\frac{\mu_2(\Lambda_m(i,j))}{\mu_2(\Lambda_m)}$ to be in
$\Lambda_m(i,j)$, and $X_l$ are iid. random variables. Using the bound
of~\eqref{eq:crossing_annulus_simpler} on
Lemma~\ref{lema:crossing_annulus} we get
\begin{align}
\PP(X_l = n)
    &= \sum_{i=1}^{m - n +1} \PP(s_l \in \Lambda_m(i, i + n - 1))
    = \sum_{i=1}^{m - n +1}
    \frac{\mu_2(\Lambda_m(i, i - n + 1))}{\mu_2(\Lambda_m)} \nonumber\\
    &\le \sum_{i=1}^{m - n +1}
    \frac{
        \mu_2(s; E_0(s) \cap \partial B(2^i) \neq \varnothing,
        E_0(s) \cap \partial B(2^{i - n + 1}) \neq \varnothing)
    }{
        \mu_2(\Lambda_m)
    }
    \le \frac{mc}{\lambda_m} \cdot 2^{-n}.
\label{eq:bound_unif_stick_intersect_n_circles}
\end{align}

Fixing any $\theta \in (0, \log 2)$ we have that
\begin{align*}
\PP\Bigl( \bigcap_{j \in [m]} N_0(j) \geq k\Bigr)
    &\le \PP\Bigl( \sum_{j=1}^{m} N_0(j) \geq mk \Bigr)
    = \PP\Bigl( \sum_{l=1}^{M} X_l \geq mk \Bigr)
    \le e^{- \theta mk} \cdot \EE\bigl[ e^{ \theta \sum_{l=1}^{M} X_l } \bigr] \\
    &= \exp\bigl[- \theta mk + u\lambda_m(\EE[e^{\theta X_l}] - 1)\bigr],
\end{align*}
by Markov inequality and the independence of $M$ and random variables $X_l$.
Estimating $\EE[e^{\theta X_l}]$
with~\eqref{eq:bound_unif_stick_intersect_n_circles}, we get
\begin{equation*}
\EE[e^{\theta X_l}]
    = \sum_{n=1}^m e^{\theta n} \PP[X_l = n]
    \le \frac{cm}{\lambda_m} \sum_{n=1}^{m}\exp[(\theta - \log 2)n ]
    \le \frac{c_\theta m}{\lambda_m}
\end{equation*}
and hence
\begin{equation*}
\PP\bigl( N_0(j) \geq k, j \in [m]\bigr)
    \le \exp\bigl[- \theta mk + uc_\theta m - u\lambda_m \bigr]
    \le \exp\Bigl[ - \Bigl(\theta - \frac{uc_\theta}{k}\Bigr) mk \Bigr].
\end{equation*}

Taking $k_0(u) = \lceil 2\frac{uc_\theta}{\theta}\rceil$, we have for
$k \geq k_0(u)$ that
$\theta - \frac{uc_\theta}{k} \geq \theta - \frac{uc_\theta}{k_0(u)}
\geq \frac{\theta}{2} > 0$.

To extend our bounds for $k < k_0(u)$, we prove it first for $k = 1$. Recall
that starting from random variables $D_j$ in \eqref{eq:def_Drj} we defined
$I_{0} := m$, $I_{j} := D_{I_{j-1}}$ and
$T := \max\{j; I_{j} \geq 1\}$. Replicating the argument in
Lemma~\ref{lema:up_bound_cap_random_vacant_circuits} for events
$\{ N_0(I_{j}) \geq 1\}$ instead of $\vCirc_0(A_{I_j})^{\comp}$
(notice that both are increasing), we get
\begin{align*}
\PP\bigl( N_0(j) \geq 1, j \in [m]\bigr)
    &\le \PP\bigl(T < \lfloor \kappa m \rfloor\bigr) +
        \PP\bigl(T \geq \lfloor \kappa m \rfloor,
        \cap_{j=1}^{\lfloor \kappa m \rfloor} N_0(I_{j}) \geq 1
        \bigr) \\
    &\le \exp[- c(u) \cdot m] +
        (1 - \exp[-u2\pi])^{\lfloor \kappa m \rfloor} \\
    &\le 2 e^{- c(u) \cdot m}.
\end{align*}
Thus, for any $k < k_0(u)$ we have
\begin{align*}
\PP\bigl( N_0(j) \geq k, j \in [m]\bigr)
    \le \PP\bigl( N_0(j) \geq 1, j \in [m]\bigr)
    \le 2 e^{-c(u) \cdot m}
    \le 2 e^{ -\frac{c(u)}{k_0(u)} mk }
\end{align*}
and we can pick constants $c(u), K_k(u) > 0$ that
verify~\eqref{eq:hypothesis_H1_part2} for all $k \in \NN$.
\end{proof}

We conclude~\eqref{eq:hypothesis_H1_IPS_recall} holds, finishing
the proof of Theorem~\ref{teo:hypothesis_H1_IPS}.
\end{proof}

As mentioned in the Introduction, 
Theorem~\ref{teo:tight_regularity_exploration_paths_2} is immediate from
the theory developed in~\cite{aizenman1999holder}, once we have
Theorem~\ref{teo:hypothesis_H1_IPS}. The lower bound on the Hausdorff dimension
of limiting paths follows from Proposition~\ref{prop:property_varnothing}.

\section*{Acknowledgments}

The authors would like to thank an anonimous referee for his careful reading.
In particular, for pointing out reference~\cite{basdevant} which led to the
proof of Proposition~\ref{prop:property_varnothing}. Most of this work was
developed in DU Phd Thesis at IMPA.  The research of D.U. had financial
support by FAPERJ, grant 202.231/2015, and FAPESP, grant 2020/05555-4. A.T. was
supported by grants “Projeto Universal” (406250/2016- 2) and “Produtividade em
Pesquisa” (304437/2018-2) from CNPq and “Jovem Cientista do Nosso Estado”,
(202.716/2018) from FAPERJ.

%%  imsart-number.bst  will be used to                     %%
%%  MR numbers will be added by VTeX.                      %%
%%                                                         %%
%%  Use \cite{...} to cite references in text.             %%
%%                                                         %%
%%%%%%%%%%%%%%%%%%%%%%%%%%%%%%%%%%%%%%%%%%%%%%%%%%%%%%%%%%%%%

% if your bibliography is in bibtex format, uncomment commands:
\bibliographystyle{plain} % Style BST file
\bibliography{ips_interface.bib}      % Bibliography file (usually '*.bib')

\end{document}